\documentclass[11pt]{article}

\usepackage[utf8]{inputenc}
\usepackage{algorithm}
\usepackage{algorithmic}
\usepackage{amsfonts}
\usepackage{amsmath,amssymb}
\usepackage{amsthm} 
\usepackage{booktabs}
\usepackage{color}
\usepackage{enumerate}
\usepackage{float}
\usepackage{graphicx}
\usepackage[colorlinks=true,linkcolor=blue,citecolor=blue,urlcolor=blue]{hyperref}
\usepackage{latexsym}
\usepackage{longtable}
\usepackage{mathtools}
\usepackage{multirow}
\usepackage{physics}
\usepackage{setspace}
\usepackage{subfigure}
\usepackage[normalem]{ulem}
\usepackage{url}
\usepackage{tikz}
\usepackage[colorinlistoftodos]{todonotes}

\newcounter{num}

\newcommand{\rnum}[1]{\setcounter{num}{#1}\roman{num}}

\definecolor{myred}{RGB}{255,50,50}         
\definecolor{myblack}{RGB}{0,0,0}           
\newcommand{\X}[1]{}
\newcommand{\red}[1]{\textcolor{black}{#1}}  
\newcommand{\blue}[1]{\textcolor{black}{#1}} 
\newcommand{\inner}[2]{\langle#1,#2\rangle}
\newcommand{\Arg}{\mathrm{Arg}}
\renewcommand{\Re}{\mathbb{R}}
\newcommand{\B}{\mathcal{B}}

\newcommand{\M}{\mathbf{M}}
\newcommand{\Q}{\mathcal{Q}}
\renewcommand{\L}{\mathcal{L}}
\newcommand{\Z}{\mathbb{Z}}

\newcommand{\dom}{\mathrm{dom}\,}

\newcommand{\supp}{\mathrm{supp}}

\newcommand{\stkout}[1]{\ifmmode\text{\sout{\ensuremath{#1}}}\else\sout{#1}\fi}
\usetikzlibrary{patterns, arrows.meta}

\newtheorem{theorem}{Theorem}[section]
\newtheorem{lemma}[theorem]{Lemma}
\newtheorem{definition}[theorem]{Definition}

\newtheorem{proposition}[theorem]{Proposition}

\numberwithin{equation}{section}

\title{Biobjective optimization with M-convex functions%
  \thanks{This work was supported by the Grant-in-Aid for Scientific
    Research (C) (JP25K15002) and Grant-in-Aid for Scientific Research (B) (JP25K03082) from Japan Society for the Promotion of Science.}
}
\author{
  Ellen H. Fukuda%
  \thanks{Graduate School of Informatics, Kyoto University, Kyoto \mbox{606--8501}, Japan.}
  \and
  Satoru Iwata%
  \thanks{Graduate School of Information Science and Technology, The University of Tokyo, Tokyo \mbox{113--8656}, Japan.}
  \thanks{Institute for Chemical Reaction Design and Discovery (ICReDD), Hokkaido University, Sapporo, Hokkaido \mbox{001--0021}, Japan.}
  \and
  Itsuki Nakagawa$^\ddag$%
}

\begin{document}

\maketitle

\begin{abstract}
  \noindent In this paper, we deal with two ingredients that, as far as we know, have not been combined until now: multiobjective optimization and discrete convex analysis. First, we show that the entire Pareto optimal value set can be obtained in polynomial time for biobjective optimization problems with discrete convex functions, in particular, involving an M$^\natural$-convex function and a linear function with binary coefficients. We also observe that a more efficient algorithm can be obtained in the special case where the M$^\natural$-convex function is M-convex. Additionally, we present a polynomial-time method for biobjective optimization problems that combine M$^\natural$-convex function minimization with lexicographic optimization. \\
  

\end{abstract}


\section{Introduction}

In multiobjective optimization, multiple objective functions have to be minimized simultaneously. In this context, the concept of Pareto optimality (or efficiency) is used, as it represents the trade-offs between the objectives.
Many approaches have been developed for multiobjective optimization, including metaheuristics~\cite{JMT02} and descent methods~\cite{FS00,FGD14}. Another major class of methods is known as scalarization, in which several parameterized single-objective optimization problems are solved in order to approximate or recover the Pareto optimal set. Among these, the weighting (or weighted sum) method selects a weighting parameter and minimizes a convex combination of the objective functions using that parameter. When the decision variables are continuous, the weighting method can recover the entire Pareto optimal set by assuming the convexity of the functions~\cite{Mie99}. 

However, this does not hold when the feasible set has a discrete structure. In fact, it is known that some Pareto optimal points cannot be obtained with the weighting method even if all the functions are linear~\cite{Ehr05,EG00}. The solutions that can be obtained through this method are called supported Pareto (or efficient) points, and finding them is equivalent to solving a series of single-objective combinatorial optimization problems. In practice, however, the number of non-supported Pareto optimal points may grow exponentially with instance size. A well-known alternative is the $\varepsilon$-constraint method, which minimizes one selected objective function while treating the others as inequality constraints. Although this method can find the entire Pareto optimal set, the resulting scalarized problem is $\mathcal{NP}$-hard in general. In fact, multiobjective problems with combinatorial structures are hard in terms of complexity. It is usually $\mathcal{NP}$-hard and $\#\mathcal{P}$-complete, respectively, for finding and counting the Pareto solutions~\cite{Ehr00}. 

Dealing with multiobjective optimization problems on matroids is similarly challenging. In fact, the corresponding decision problem is also $\mathcal{NP}$-complete in general~\cite{Ehr96}. Despite the intractability of these problems, Gorski et al.~\cite{GKS23} recently showed that a specific biobjective problem on matroids admits only supported Pareto optimal solutions, and that the Pareto optimal value set is connected, in a sense related to the underlying discrete structure~\cite{GKR11}. In their setting, the two objectives are linear: one with nonnegative integer coefficients (which is also usual in multiobjective combinatorial optimization), and the other restricted to binary coefficients. With this particular structure of the Pareto optimal value set, they proposed an efficient method that enumerates this entire set.

In this paper, we also consider biobjective optimization problems, but one of the objectives is M-convex or, more generally, M$^\natural$-convex, extending the problem studied in~\cite{GKS23}. We further generalize the problem by considering optimization over integer lattice points instead of restricting to matroids. To the best of our knowledge, this is the first work that considers these discrete convex functions in the context of multiobjective optimization. We recall that M-convexity and M$^\natural$-convexity are key instruments in discrete convex analysis, introduced in~\cite{Mur96a,MS99}. An example of M-convex functions arises in network flow problems, while M$^\natural$-convexity has been shown to be equivalent to the economic concept of gross substitutes~\cite{FY03,Shi15}. More recently, the applicability of these functions has extended beyond economics and game theory to practical domains such as bicycle-sharing systems~\cite{Shi22}.

This study focuses on developing efficient algorithms for biobjective optimization over integer lattice points involving M-convex and M$^\natural$-convex functions. Since the weighting method is sufficient if one seeks only supported Pareto optimal solutions, we focus instead on cases where the entire Pareto optimal value set can be obtained. 

We first consider the biobjective problem with an M$^\natural$-convex function and a linear function with binary coefficients. Our proposed method is based on the works of Gorski et al.~\cite{GKS23} and Takazawa~\cite{Tak23}. More specifically, we reformulate the original problem as a sequence of single-objective optimization problems, where the M$^\natural$-convex function serves as the objective, and the linear function is incorporated as an equality constraint. 
We show that the whole Pareto optimal value set can be obtained in polynomial time. Based on this, we further specify the problem by considering the two objectives as an M-convex function and a linear function with binary coefficients. By exploiting the properties of M-convex functions, we refine the algorithm to achieve improved computational complexity. 

Lastly, we consider biobjective problems involving M$^\natural$-convex minimization and lexicographic optimality. In particular, we show that the reformulation with the $\varepsilon$-constraint method, using the order induced by the lexicographic cone, guarantees Pareto optimal values. By using the structure of the lexicographic cone, we also present a polynomial-time algorithm based on valuated matroid intersection that can enumerate this set of Pareto optimal values without requiring any filtering.

The outline of this paper is as follows. In Section~\ref{sec:preliminaries}, we define some notations, explaining also the basics about generalized matroids, discrete convex functions, multiobjective optimization, as well as related works. In Section~\ref{sec:MnatBB}, we discuss the biobjective optimization problem, where one objective is M$^\natural$-convex, and the other one is linear with binary coefficients, proposing also an algorithm to solve it. In Section~\ref{sec:MBB}, we consider the same biobjective problem but replace the M$^\natural$-convex function with an M-convex one. An improved version of the algorithm is then proposed for this problem. In Section~\ref{sec:MLB}, we investigate biobjective problems with M$^\natural$-convex functions and lexicographic objectives. Finally, we conclude the paper in Section~\ref{sec:conclusion}, with some possible future works.


\section{Preliminaries}
\label{sec:preliminaries}

We start this section with the following notations, which will be used throughout the paper. We use $\Z$ and $\Re$ to denote the set of all integers and real numbers, respectively. The notation $\Z^E$ (or $\Re^E$) means the integer (or real) space with coordinates indexed by the elements of a finite set $E$. If~$E$ contains $n$~elements, $\Z^E$ (or $\Re^E$) is often written as $\Z^n$ (or $\Re^n$). 
For a vector~$x$, its $e$-th component is written as $x(e)$. 
The $\ell_1$-norm and infinity norm of~$x$ are denoted by $\norm{x}_1$ and $\norm{x}_{\infty}$, respectively. We also denote by $\supp(x)$ the set of elements with positive components, i.e.,
\[
  \supp(x):=\{e\mid x(e)>0\}.
\]
Moreover, for given vectors~$x$ and~$y$, we write
$x \leq y$ ($x < y$) if and only if $x(e) \leq y(e)$ ($x(e) < y(e)$) hold for all $e$.
Here, we deal with extended-valued convex functions, and thus we use the notations 
$\overline{\Re} := \Re \cup \{ +\infty \}$ and 
$\underline{\Re} := \Re \cup \{ -\infty \}$. 
For a function~$f$ defined on a domain~$\mathcal{X}$, its \emph{effective domain} is defined by 
$\dom \, f := \left\{ x \in \mathcal{X} \mid f(x) \in\Re \right\}$.
For a nonempty finite set $E$ and a subset $S \subseteq E$, we write
\[
  x(S) := \sum_{e \in S} x(e),
\] 
and assume that $x(\varnothing) = 0$. 
A set $S \subseteq E$ can be uniquely associated with a vector 
$\chi_S \in \{0,1\}^E$ of dimension $|E|$. Such a vector $\chi_S$ is called the \emph{characteristic vector} of $S$, and is defined by
\begin{equation*}
  \chi_S(e) := \left\{
    \begin{alignedat}{2}
      1, \quad & \mbox{if } e \in S, \\
      0, \quad & \mbox{if } e \in E \setminus S.
    \end{alignedat}
 \right.
\end{equation*}
For $e \in E$, we abbreviate $\chi_{\{e\}}$ as $\chi_e$.

\subsection{Generalized matroids}

The concepts of matroids and polymatroids are closely related to discrete convex functions that will be reviewed in Section~\ref{sec:discrete_convex}. Here, we only recall the basic notions of matroids and refer the reader to~\cite{KV12,Oxl06} for further details. 

Let $E$ be a finite \emph{ground set} and $\B$ be a nonempty family of subsets of $E$. The pair $\M = (E,\B)$ is a \emph{matroid} if and only if the following property holds. \\[5pt]
\begin{tabular}{@{}ll}
  \textbf{(B)} & For any $B_1, B_2 \in \B$ and $u \in B_1 \setminus B_2$, there exists $v \in B_2 \setminus B_1 $ such that\\
  & $ B_1 \setminus \{u\} \cup\{v\}\in\B$ and $B_2\cup\{u\}\setminus\{v\}\in\B$.\\[5pt]
\end{tabular}

\noindent Here, $\mathcal{B}$ is called \emph{base family}, i.e., a collection of all bases of $\M$. The above property is known as the \emph{simultaneous exchange axiom}, introduced by Brualdi~\cite{Bru69}. In this paper, we adopt this characterization of matroids, rather than the standard exchange property, or the definition via independence system with the augmentation property. We denote by $\rank(\M)$ the \emph{rank} of $\M$, i.e., the common cardinality of the bases~\cite[Proposition~1.2.1]{Oxl06}.


Assume that the ground set $E$ is partitioned into $E_1,\dots,E_m$, that is, $E := \bigcup_{k=1}^m E_k$ and $E_k \cap E_{l} = \varnothing$ for all $k \neq l$. Given nonnegative integers $\eta_k$ with \red{$\eta_k\leq |E_k|$ for $k=1,\dots,m$, consider the family $\B$ of subsets $X \subseteq E$ such that $|E_k \cap X|=\eta_k$ hold for all $k$. The set family $\B$ satisfies \textbf{(B)}, and the matroid $\M=(E,\B)$ thus constructed is called a \emph{partition matroid}}. 

The concept of generalized polymatroids (also known as g-polymatroids), introduced by Frank~\cite{Fra84} \red{and Hassin~\cite{Has82}}, includes polymatroids and base polyhedron as special cases; see also~\cite{FT88,Fuj84} for additional discussions. Here we describe 
the concept of generalized matroids (g-matroids)~\cite{Tar85}, which are special cases of generalized polymatroids.

Let $\mathcal{P}$ be a nonempty family of subsets of $E$.
The pair $\M = (E,\mathcal{P})$ is a 
\emph{g-matroid} if and only if the following property holds. \\[5pt]
\begin{tabular}{@{}ll}
  \textbf{(P)} & For any $P_1, P_2 \in \mathcal{P}$ and $v \in P_1 \setminus P_2$, one of the following conditions holds:\\
  & (i) \,$P_1 \setminus\{v\}\in\mathcal{P}$ and 
  $P_2\cup\{v\}\in\mathcal{P}$, \\ 
  & (ii) there exists $u \in P_2 \setminus P_1 $ such that $ P_1 \setminus\{v\}\cup\{u\} \in \mathcal{P}$ \\ & \quad\:\: and $P_2 \cup\{v\}\setminus\{u\}\in\mathcal{P}$.
\end{tabular}

\noindent As in the case of matroids, we consider this characterization of g-matroids, based on the above exchange properties. Alternatively, however, one may use another exchange property due to Tardos~\cite{Tar85}, or the fact that a g-polymatroid can be obtained as a projection of a base polyhedron, as shown by Fujishige~\cite{Fuj84}.


\subsection{Discrete convex functions}
\label{sec:discrete_convex}

This work lies within the paradigm of discrete convex analysis, dealing with the so-called M-convex functions, introduced by Murota~\cite{Mur96a}. We also deal with M$^\natural$-convex functions, proposed by Murota and Shioura~\cite{MS99}, that contain the class of M-convex functions. We refer the reader to~\cite{Mur98,Mur03} for more details related to discrete convex analysis.

A function $f \colon \Z^E \rightarrow \overline{\Re}$ with $\dom f \ne \varnothing$ is called \emph{M-convex} if for each $x, y \in \dom f$ and $u \in \supp(x-y)$, there exists $v \in \supp(y-x)$ such that
\[ 
  f(x) + f(y) \geq f(x-\chi_u+\chi_v) + f(y+\chi_u-\chi_v).
\]
This is known as the \emph{exchange axiom}. Note that $x-\chi_u+\chi_v \in \dom f$ and $y+\chi_u-\chi_v \in \dom f$ are implicitly assumed.
Moreover, a function $f \colon \Z^E \rightarrow \overline{\Re}$ with $\dom f \ne \varnothing$ is called
\emph{M$^\natural$-convex} if for each $x,y \in\dom f$ and $u \in \supp(x-y)$, one of the following conditions holds:
\begin{enumerate}[(i)]
  \item[(i)] $f(x) +f(y) \geq f(x-\chi_u) + f(y+\chi_u)$,
  \item[(ii)] there exists $v \in \supp(y-x)$ such that $f(x) +f(y) \geq f(x-\chi_u+\chi_v) + f(y+\chi_u-\chi_v)$.
\end{enumerate}
The above definitions show that M$^\natural$-convexity generalizes M-convexity. 
Although the two concepts are essentially equivalent~\cite[Theorem~6.3]{Mur03}, 
we distinguish them in this study, as restricting to M-convex functions can lead to improved computational complexity.

We also recall that the set of integer lattice points in an integral base polyhedron (g-polymatroid) is called \emph{M-convex} (\emph{M$^\natural$-convex}) \emph{set}. It is known that a set being M-convex (M$^\natural$-convex) is equivalent to its indicator function being M-convex (M$^\natural$-convex). Moreover, the effective domain of an M-convex (M$^\natural$-convex) function is an \red{M-convex (M$^\natural$-convex)} set. 
The property that \red{the set family corresponding to an $0$-$1$ M-convex (M$^\natural$-convex) set} must satisfy is nothing other than the simultaneous exchange axiom~\textbf{(B)} (\textbf{(P)}), given in the previous section. 

Let us now present some properties involving the above functions. The sum of M-convex (M$^\natural$-convex) functions is not necessarily M-convex (M$^\natural$-convex), and such functions are called M$_2$-convex (M$^\natural_2$-convex) functions.  However, from the exchange axiom, it is easy to see that the sum of an M$^\natural$-convex function and a linear function is M$^\natural$-convex.

\subsection{Multiobjective optimization}

Besides discrete convex analysis, in this work, we also deal with optimization problems with two objective functions. For completeness, we discuss here the more general case with multiple objectives. We refer to~\cite{Ehr96, Ehr05} for more details related to multiobjective optimization, in particular, with combinatorial structure.
Let us then consider the following optimization problem with $m$~objective functions:
\begin{align*}
\min \quad & f(x) := \left(f_1(x),\,f_2(x),\dots,f_m(x)\right), \\
\mbox{s.t.} \quad & x \in \dom \, f.
\end{align*}
In multiobjective optimization, it is not possible, in general, to minimize all the objective functions simultaneously. So, the concept of Pareto optimality is used instead. 

If $f(y) \leq f(x)$ and $f(y) \neq f(x)$, then we say that $y$ \emph{dominates} $x$. 
A point $x$ is called a \emph{Pareto optimal} solution if there exists no feasible point that dominates $x$, i.e.,
\[ 
  x \mbox{ is Pareto optimal} \quad \Leftrightarrow \quad 
  \nexists\, y\in \dom f \mbox{ s.t. } f(y) \leq f(x),\ f(y) \neq f(x).
\]
In other words, an improvement in one objective function cannot be achieved without compromising at least one other objective. 
Furthermore, $y$ is said to \emph{strongly dominate} $x$ if $f(y) < f(x)$ holds. A point $x$ is called a \emph{weakly Pareto optimal} solution if there is no feasible point that strongly dominates $x$, i.e.,
\[ 
  x \mbox{ is weakly Pareto optimal} \quad \Leftrightarrow \quad 
  \nexists\, y\in \dom f \mbox{ s.t. } f(y) < f(x).
\]
The set of (weakly) Pareto optimal solutions is called the (\emph{weakly}) 
\emph{Pareto frontier}. Clearly, all Pareto optimal points are weakly Pareto,
but the converse does not necessarily hold. We also refer to $\{ f(x) \mid x \text{ is (weakly) Pareto optimal} \}$ as the (\emph{weakly}) \emph{Pareto optimal value set} (or (weakly) non-dominated set).

One of the most common approaches for solving multiobjective optimization is the 
\emph{scalarization method}, where a set of scalar-valued problems is solved repeatedly. Among them, the \emph{weighting} (or \emph{weighted sum}) \emph{method} minimizes a convex combination of the objective functions. While this method can generate Pareto optimal points, it cannot obtain all such points. In the context of multiobjective combinatorial optimization, those that cannot be reached are known as \emph{non-supported Pareto} optimal points, which may grow exponentially with instance size. \red{Here, we consider the \emph{supported Pareto} optimal points to be those that can be generated by the weighting method with strictly positive weights. In our context, they correspond to the values $z\in\mathrm{conv}(Z_\mathrm{P})$ such that $\mathrm{conv}(Z_\mathrm{P}) \cap (z-\Re^m_+) = \{z\}$, 
where $Z_\mathrm{P}$ is the Pareto optimal value set, $\mathrm{conv}(Z_\mathrm{P})$ denotes the convex hull of $Z_\mathrm{P}$, and $\Re^m_+$ is the nonnegative orthant of $\Re^m$. See~\cite{KS25} for a discussion of the differences among various definitions of supported Pareto solutions.}

The \emph{$\varepsilon$-constraint method}, on the other hand, consists of choosing one objective as the function to be minimized, moving the others to the constraints. The optimal solution to the resulting problem is known to be weakly Pareto optimal, and in particular, if the solution is unique, it corresponds to a Pareto optimal point. Using this approach, all the Pareto frontier can be obtained. However, since the constraints are of the knapsack type, the scalarized problem is generally $\mathcal{NP}$-hard.

\subsection{A specific biobjective optimization on matroids}

Although multiobjective optimization on a matroid $\M$ is complex in general, Gorski et al.~\cite{GKS23} have shown that the whole Pareto optimal value set can be computed in polynomial time in the biobjective case, where the objective functions over the set of bases $\B$ are defined by an integer-valued function $c \colon E \rightarrow \mathbb{Z}$ and a binary-valued function $b \colon E\rightarrow \{0,1\}$. More precisely, the biobjective problem is given by
\begin{align}
  \min \quad & (c(B),\,b(B)) \label{prob:LBB} \\
  \mbox{s.t.} \quad & B \in \mathcal B, \nonumber
\end{align}
where $c(B) := \sum_{e \in B} c(e)$ and $b(B) := \sum_{e \in B} b(e)$. The above problem is then replaced by a set of problems
\begin{align*}
  \min \quad &c(B) \\
  \mbox{s.t.} \quad &b(B) = k,  \\
  \quad &B \in \mathcal B,
\end{align*}
where $k \in \{0,\ldots, \mathrm{rank}(\M)\}$. The above subproblem is, in fact, a special case of matroid intersection, where the equality constraint is commonly referred to as a color-induced budget constraint.
For each $k$, one can compute in polynomial time a basis $B$ that solves the above problem (see Gabow and Tarjan~\cite{GT84}, and Gusfield~\cite{Gus84}).

In addition, recently, Takazawa~\cite{Tak23} showed that the more general problem below can also be solved in polynomial time for each $k$:
\begin{align*}
  \min \quad & g(x) \\
  \mbox{s.t.} \quad & \langle b,x\rangle = k,  \\
  \quad & x \in \dom g,
\end{align*}
where $g \colon \Z^E \rightarrow  \overline{\Re}$ is an
M$^\natural$-convex function and $\langle b,x \rangle:=\sum_{e\in E}b(e)x(e)$.  
Based on this, it is expected that the biobjective optimization problem involving an M$^\natural$-convex  function $g$ and a linear function $b$ with binary coefficients can be solved efficiently. This is indeed the topic of the next section.


\section{\texorpdfstring{Biobjective optimization with M$^\natural$-convex and binary linear functions}{Biobjective optimization with M-natural-convex and binary linear functions}}
\label{sec:MnatBB}

In this section, we consider the following biobjective optimization problem\footnote{We call the problem M$^\natural$BB, because it is a 
\underline{b}i-objective optimization problem involving an \underline{M$^\natural$}-convex function and a linear function with \underline{b}inary coefficients. The names of the other problems are similarly defined.}:
\begin{align}
  \min \quad & (g(x),\,\langle b,x\rangle) \label{prob:MnatBB} \tag{M$^\natural$BB} \\
  \mbox{s.t.} \quad & x \in \dom g, \nonumber
\end{align}
where $g \colon \Z^E \rightarrow  \overline{\Re}$ is an
M$^\natural$-convex function and $\langle b,x\rangle := \sum_{e\in E} b(e)x(e)$ with a binary function $b\colon E \rightarrow \{0,1\}$. Note that this problem generalizes~\eqref{prob:LBB}.

Let us partition the ground set $E$ based on the value of $b(e)$ for each $e \in E$, i.e.,
\begin{equation}
  \label{eq:partition_E}
  E_0 := \{e \in E \mid b(e) = 0\}, \quad
  E_1 := \{e \in E \mid b(e) = 1\}.
\end{equation}
We also consider the following partition of $\mathcal{P}:=\dom g$, based on the value of $\langle b,x\rangle$,~i.e., 
\[
  \mathcal{P}_i := \{ x\in\mathcal{P}\mid \langle b,x\rangle= i\}.
\]
Then, for each $\mathcal{P}_i$, we can define the family of optimal sets with respect to $g$ by
\[
  \mathcal T_i := \{ x \in \mathcal P_i \mid  g(x) \leq g(y)
  \mbox{ for all } y \in \mathcal{P}_i\}.
\]
Each of $\mathcal{T}_i$ corresponds to the set of optimal solutions of the following subproblem:
\begin{align}
  \min \quad & g(x) \label{prob:MnatBBi} \tag{M$^\natural$BB$_{=i}$} \\
  \mbox{s.t.} \quad & x \in \mathcal P, \nonumber \\
  & \langle b,x\rangle = i. \nonumber
\end{align}
It is important to note that the solution of~\eqref{prob:MnatBBi} is not necessarily a Pareto optimal point of~\eqref{prob:MnatBB}, since, for example, the optimal value of $\mathrm{(M^\natural BB_{=4})}$ may be lower than that of $\mathrm{(M^\natural BB_{=5})}$. If the equality constraint is replaced with the inequality $\langle b,x\rangle \leq i$, then the approach is equivalent to the $\varepsilon$-constraint method, and thus, the solutions of those problems are weakly Pareto. But here we note that~\eqref{prob:MnatBBi} is easier to solve, and because of the special structure of the problem, the Pareto optimal set can be derived from its solutions. 

\begin{definition}
  \label{def:transition}
  A \emph{transition} $(u,v)$ with respect to $x \in \mathcal{P}$ is a pair of $u \in E_1$ and $v \in E_0 \cup \{z\}$ satisfying $x-\chi_u+\chi_v \in \mathcal{P}$, where $z$ is a dummy symbol with $\chi_z$ being the zero vector. 
\end{definition}

The definition above can be seen as an extension of \emph{swap}, introduced  in~\cite{GKS23}, which is restricted to base families. In that setting, the case $v = z$ does not appear. We now define the cost of a transition below.

\begin{definition}
  For an M$^\natural$-convex function $g$, the \emph{cost of a transition} $(u,v)$ from $x \in \mathcal P$ is defined by
  \[ 
    g(x; u,v) := g(x-\chi_u+\chi_v) - g(x).
  \]
  Also, a \emph{minimum transition} with respect to $x$ is defined by a transition $(u,v)$ such that
  \[ 
    g(x; u,v) \leq g(x; u',v')
  \]
  for any transition $(u',v')$. 
\end{definition}

The following lemma will be necessary in the subsequent analysis. Before that, let us denote the following indices:
\[
  \underline{k} := \min \{ k \mid \mathcal T_k \neq \varnothing\} \quad
  \mbox{and} \quad \bar k := \max \{ k \mid \mathcal T_k \neq \varnothing\}.
\]

\begin{lemma}[{\cite[Lemma~2]{Tak23}}]
  \label{lem:T_k_nonempty}
   For all $k$ with $\underline k \leq k \leq \bar k$ we have $\mathcal T_k \neq \varnothing$.
\end{lemma}

\begin{lemma}
\label{lem:no-transition}
If there exists no transition from $x\in \mathcal{P}$, then $\underline{k}=\langle b,x\rangle$ holds. 
\end{lemma}
\begin{proof}
Suppose to the contrary that $\langle b,x\rangle>\underline{k}$. 
Then Lemma~\ref{lem:T_k_nonempty} implies that there exists $y\in\mathcal{P}$ such that $\langle b,y\rangle=\langle b,x\rangle -1$. Assume that $y$ reaches the minimum value of $\|x-y\|_1$ among those that satisfy this condition. There must exist $u\in\supp(x-y)\cap E_1$. By the definition of M$^\natural$-convex functions, $g(x)+g(y)\geq g(x-\chi_u)+g(y+\chi_u)$ holds, or there exists $v\in\supp(y-x)$ such that $g(x)+g(y)\geq g(x-\chi_u+\chi_v)+g(y+\chi_u-\chi_v)$. In the former case, $(u,z)$ is a transition. In the latter case, if $v\in E_1$, then $y':=y+\chi_u-\chi_v$ satisfies $\langle b,y'\rangle=\langle b,y\rangle=\langle b,x\rangle-1$ and $\|x-y'\|_1=\|x-y\|_1-2$, which contradicts our choice of $y$. Therefore, we have $v\in E_0$, which means that $(u,v)$ is a transition. 
\end{proof}

The following result can be seen as a combination of the previous results by Gabow and Tarjan \cite[Theorem~3.1]{GT84} and Takazawa \cite[Lemma~3]{Tak23}.

\begin{proposition}
  \label{prop:min_trans}
  Let $k$ be an integer such that $\underline k +1 \leq k \leq \bar k$, and take $x_k \in \mathcal T_k$. Then, by applying the minimum transition $(u^\ast,v^\ast)$, we have $x_k -\chi_{u^\ast} +\chi_{v^\ast} \in \mathcal T_{k-1}$.
\end{proposition}

\begin{proof}
  From Lemma~\ref{lem:T_k_nonempty}, we know that $\mathcal T_{k-1} \neq \varnothing$.
  Hence, we can select $x' \in \mathcal T_{k-1}$ such that $\|x'-x_k\|_1$ is minimized.
  Since $x_k \in \mathcal P_k$ and $x' \in \mathcal P_{k-1}$, we obtain $\langle b,x'\rangle < \langle b,x_k \rangle$. Thus, there exists $u \in E_1$ with $u\in \supp(x_k-x')$. From the M$^\natural$-convexity of $g$, one of the following two cases holds:
  \begin{enumerate}[(i)]
    \item[(i)] $g(x_k)+g(x')\geq g(x_k-\chi_u)+g(x'+\chi_u)$, 
    \item[(ii)] there exists $v \in\supp(x'-x_k)$ such that
    $g(x_k)+g(x')\geq g(x_k-\chi_u+\chi_v)+g(x'+\chi_u-\chi_v)$.
  \end{enumerate}
  Let us analyze the above two cases separately.

  First, assume that (i) holds. Since we have $x_k - \chi_u \in \mathcal P_{k-1}$ and $x' +\chi_u \in \mathcal P_{k}$, by the definitions of $\mathcal T_k$ and $\mathcal T_{k-1}$, we have 
  \begin{equation}
    \label{eq:case_i_g}  
    g(x') \leq g(x_k -\chi_u) \quad \mbox{and} \quad g(x_k) \leq g(x' +\chi_u),
  \end{equation}
  which together with (i) implies that 
  \[ 
    g(x_k) + g(x') = g(x_k - \chi_u) + g(x' +\chi_u).
  \] 
  If $g(x')<g(x_k-\chi_u)$, then the above equality shows that $g(x_k)>g(x'+\chi_u)$, which contradicts~\eqref{eq:case_i_g}. Therefore, $g(x') = g(x_k -\chi_u)$ and $g(x_k) = g(x'+\chi_u)$ hold, and consequently $x_k - \chi_u \in \mathcal T_{k-1}$.

  Next, assume that (ii) holds. We will show that $b(v)=0$ in this case. 
  Assuming instead that $b(v) = 1$, we obtain $x_k -\chi_u+\chi_v \in \mathcal P_k$ and $x'+\chi_u-\chi_v \in \mathcal P_{k-1}$. Thus, once again from the definitions of $\mathcal T_k$ and $\mathcal T_{k-1}$ we have
  \begin{equation}
    \label{eq:case_ii_g}  
    g(x_k) \leq g(x_k-\chi_u+\chi_v) \quad \mbox{and} \quad g(x') \leq g(x' +\chi_u-\chi_v), 
  \end{equation}
  which together with (ii) implies
  \begin{equation}
    \label{eq:equal_g}
    g(x_k)+g(x')=g(x_k-\chi_u+\chi_v)+g(x'+\chi_u-\chi_v).      
  \end{equation}
  Assuming $g(x_k) < g(x_k -\chi_u +\chi_v)$, the above equality gives $g(x')>g(x'+\chi_u-\chi_v)$, which contradicts~\eqref{eq:case_ii_g}. Thus, we obtain $g(x_k) = g(x_k -\chi_u +\chi_v)$ and $g(x') = g(x' +\chi_u -\chi_v)$, 
  that is, $x' +\chi_u - \chi_v \in \mathcal T_{k-1}$. Since $\norm{x' +\chi_u - \chi_v - x_k} = \norm{x'-x_k} - 2$, it contradicts our choice of $x'$. We then conclude that $b(v) = 0$.

  Since in this case $x'+\chi_u-\chi_v\in \mathcal P_k$ and $x_k -\chi_u +\chi_v \in \mathcal P_{k-1}$, we have
  \[
    g(x_k) \leq g(x'+\chi_u-\chi_v) \quad \mbox{and} \quad g(x') \leq g(x_k-\chi_u+\chi_v).
  \]
  Once again, adding the above two expressions, we obtain an inequality, which together with~(ii) implies~\eqref{eq:equal_g}. By the same argument as before, this leads to $g(x_k) = g(x' +\chi_u -\chi_v)$ and $g(x') = g(x_k -\chi_u +\chi_v)$, and thus, $x_k -\chi_u +\chi_v \in \mathcal T_{k-1}$.

  From the cases (i) and (ii) above, we have shown that there exists a transition
  $(u,v)$ with $u \in E_1$ and $v \in E_0\cup\{z\}$ such that $x_k -\chi_u +\chi_v \in \mathcal T_{k-1}$. Then, for a minimum transition $(u^\ast,v^\ast)$, since 
  \begin{align*}
     \: g(x_k;u^\ast,v^\ast) \leq g(x_k;u,v) 
    \Leftrightarrow & \: g(x_k- \chi_{u^\ast} +\chi_{v^\ast}) - g(x_k) \leq g(x_k -  \chi_u+ \chi_v)  - g(x_k) \\
    \Leftrightarrow & \: g(x_k- \chi_{u^\ast} +\chi_{v^\ast}) \leq g(x_k - \chi_u + \chi_v) 
  \end{align*} 
  holds, we conclude that $x_k - \chi_{u^\ast} +\chi_{v^\ast} \in \mathcal T_{k-1}$.
\end{proof}

The above result shows that given an element $x^k \in \mathcal T_k$ with $\underline k +1 \leq k \leq \bar k$, one can always apply a transition to obtain an element in $\mathcal T_{k-1}$. This means that as we decrease the number $k$ through transitions, $\langle b,x_k\rangle$ decreases monotonically. The result below shows that the costs of the minimum transitions are also non-increasing. It generalizes \cite[Theorem 4.6]{GKS23}, which considers the particular case of a linear function, instead of an M$^\natural$-convex function.

\begin{proposition}
  \label{prop:min_trans_ineq}
  Assume that $\underline{k} + 2 \leq \bar{k}$. \red{For each $k \in \{ \underline k +2 ,\,\ldots,\bar k\}$, let $(u_k,v_k)$ be the minimum transition from $x_k \in \mathcal T_k$ to $x_{k-1} \in \mathcal T_{k-1}$, i.e., $x_{k-1} := x_k -\chi_{u_k} +\chi_{v_k}$. Then,  $g(x_k;u_k,v_k) \leq g(x_{k-1};\,u_{k-1},\,v_{k-1})$ holds.}
\end{proposition}

\begin{proof}
  \red{From the definition of transition costs, we have $g(x_k;u_k,v_k)=g(x_{k-1})-g(x_k)$ and $g(x_{k-1};u_{k-1},v_{k-1})=g(x_{k-2})-g(x_{k-1})$. Hence, it suffices to prove that 
  \begin{equation} \label{eq:g_bound}
      g(x_k)+g(x_{k-2})\geq 2g(x_{k-1})
  \end{equation}
  holds for $k\in\{\underline{k}+2,\ldots,\overline{k}\}$.}
  
We will now obtain a lower bound for the \red{left-hand side of \eqref{eq:g_bound}} by using the M$^\natural$-convexity of $g$.
\red{Recall} that $(x_{k-1},\,x_{k-2}) = (x_k - \chi_{u_k} +\chi_{v_k},\,x_k -\chi_{u_k} +\chi_{v_k} -\chi_{u_{k-1}} +\chi_{v_{k-1}})$. Since $u_{k-1}, u_k \in \supp(x_k-x_{k-2})$ and \red{$\{v_{k-1}, v_k\}=\supp(x_{k-2}-x_k)\cup\{z\}$}, we have
    \[
      g(x_k) + g(x_{k-2}) \geq
      \min\left\{
        \begin{alignedat}{1}
          g(x_k - \chi_{u_{k-1}}) + g(x_k -\chi_{u_k} +\chi_{v_k} +\chi_{v_{k-1}}), \\
          g(x_k - \chi_{u_{k}}) + g(x_k +\chi_{v_k} -\chi_{u_{k-1}} +\chi_{v_{k-1}}), \\          
          g(x_k - \chi_{u_{k-1}} + \chi_{v_k}) + g(x_k -\chi_{u_k} +\chi_{v_{k-1}}), \\
          g(x_k - \chi_{u_{k-1}} + \chi_{v_{k-1}}) + g(x_k -\chi_{u_k} +\chi_{v_k})
        \end{alignedat}
      \right\}.
    \]
  From the definition of transition, observe that the term $x_k - \chi_{u_{k-1}}$, for instance, belong to~$\mathcal P_{k-1}$. Similarly, all the arguments of~$g$ appearing on the right-hand side of the \red{above inequality} belong to~$\mathcal P_{k-1}$. Since $x_{k-1} \in \mathcal T_{k-1}$, which is guaranteed from Proposition~\ref{prop:min_trans}, 
 we obtain \eqref{eq:g_bound} from the definition of~$\mathcal T_{k-1}$.
\end{proof}

The following result shows that \red{in the process of} decreasing $k$ through transitions, once the objective value $g$ increases, it continues to increase, indicating the emergence of Pareto optimal points. 

\begin{theorem}
  \label{theo:Pareto_g}
  Let $\{x_k\}_{k=\underline k}^{\bar k}$ with $x_k \in \mathcal T_k$ be obtained by minimum transitions and assume that $\underline k+2 \leq \bar k$. Suppose that there exists an index $j \in \{\underline k +1,\ldots,\bar k\}$ such that $g(x_j) < g(x_{j-1})$. Then, for all $i \in \{ \underline k +1,\ldots,j\}$, we have $g(x_i) < g(x_{i-1})$.
\end{theorem}

\begin{proof}
  For each $i \in \{ \underline k +1,\ldots,j\}$, it follows from Proposition~\ref{prop:min_trans_ineq} that 
  \begin{align*}
    g(x_{i-1}) - g(x_i) &= g(x_i;\, u_i,\,v_i) \\
    &\geq g(x_j;\, u_j,\,v_j) \\
    &= g(x_j -\chi_{u_j} +\chi_{v_j}) - g(x_j)\\
    &= g(x_{j-1}) - g(x_j) >0,
  \end{align*}
  which completes the proof.
\end{proof}

From the above result, we can derive the following theorem related to the Pareto frontier. Note that it extends the results of~\cite[Theorem~4.7]{GKS23}.

\begin{theorem}
  \label{theo:Pareto}
  Let $\{x_k\}_{k=\underline k}^{\bar k}$ with $x_k \in \mathcal T_k$ be obtained by minimum transitions and assume that $\underline k+1 \leq \bar k$. Suppose that there exists an index $j \in \{ \underline k +1,\ldots,\bar k\}$ such that $g(x_j) < g(x_{j-1})$. Then, each element of $\{ x_k \}_{k=\underline k}^j$ is a supported Pareto optimal point of~\eqref{prob:MnatBB}.
\end{theorem}

\begin{proof}
  If $\underline k+1 = \bar k$, then $j$ is equal to $\bar{k}$. It follows from $g(x_{\bar k}) < g(x_{\underline k})$ and $\langle b, x_{\bar k}\rangle > \langle b,x_{\underline k}\rangle$ that $x_{\underline k}$ and $x_{\bar k}$ are both supported Pareto optimal.

  Let us consider the case that $\underline k+2 \leq \bar k$. Assume, for the purpose of contradiction, that $x_k$ is not Pareto optimal for some $k \in \{ \underline k,\ldots,j\}$. In this case, although there exists an $x' \in \mathcal{P}$ that dominates $x_k$, \red{it follows from Proposition \ref{prop:min_trans} that} the point $x_k$ attains the minimum value of $g$ under the constraint $\langle b,x\rangle = k$, and thus
  \[ 
    \langle b,x'\rangle < k \quad \mbox{and} \quad g(x') \leq g(x_k)
  \]
  hold. Now, if $\langle b,x'\rangle =k'$, then the first inequality above gives $k' \in \{ \underline k,\ldots, k-1\}$. From Theorem~\ref{theo:Pareto_g}, we have $g(x_k) < g(x')$, which contradicts the second inequality above. Therefore, $x_k$ is Pareto optimal.

   \red{We now intend to show} that $x_k$ is supported Pareto optimal, with $k \in \{ \underline{k}+1,\dots,j-1\}$. \red{By Proposition~\ref{prop:min_trans}, the point $x_k$ attains the minimum value of $g$ under the constraint $\inner{b}{x}=k$. Then Proposition~\ref{prop:min_trans_ineq} and Theorem~\ref{theo:Pareto_g} imply that $g(x_{k-1})-g(x_k)\geq g(x_k)-g(x_{k+1})>0$ holds. Thus, $x_k$ is supported Pareto optimal for $k \in \{ \underline{k}+1,\dots,j-1\}$.}
\end{proof}

The last theorem shows that problem~\eqref{prob:MnatBB} has a particular Pareto optimal value set, corresponding only to supported Pareto optimal points. From the above discussions, we can observe in Fig.~\ref{fig:border_points}, how Pareto optimal solutions can be obtained. 

\begin{figure}[ht]
  \begin{center}
  \begin{tikzpicture}
    \draw[-{Latex}, thick] (-0.5,0) -- (5,0);          
    \node at (5.5,0) {$g(x)$};                         
    \draw[-{Latex}, thick] (0,-0.5) -- (0,5);          
    \node at (-0.7,4.8) {$\inner{b}{x}$};              
    \draw[dashed] (0,0.75) -- (4.8,0.75);              
    \node at (-0.6,0.75) {$k-2$};                      
    \node at (3.7,1) {$x_{k-2}$};                      
    \draw[ultra thick, red] (3.25,0.75) -- (1.75,1.5); 
    \filldraw (3.25,0.75) circle (2pt);                
    \draw[dashed] (0,1.5) -- (4.8,1.5);                
    \node at (-0.6,1.5) {$k-1$};                       
    \draw[fill=white] (4,1.5) circle (2pt);            
    \node at (2.2,1.75) {$x_{k-1}$};                   
    \draw[ultra thick, red] (1.75,1.5) -- (1,2.25);    
    \filldraw (1.75,1.5) circle (2pt);                 
    \draw[dashed] (0,2.25) -- (4.8,2.25);              
    \node at (-0.3,2.25) {$k$};                        
    \filldraw (1,2.25) circle (2pt);                   
    \draw[fill=white] (3.25,2.25) circle (2pt);        
    \node at (1.4,2.5) {$x_k$};                        
    \draw[thick] (1,2.25) -- (1,3);                    
    \draw[dashed] (0,3) -- (4.8,3);                    
    \node at (-0.6,3) {$k+1$};                         
    \filldraw (1,3) circle (2pt);                      
    \draw[fill=white] (2.5,3) circle (2pt);            
    \draw[fill=white] (3.75,3) circle (2pt);           
    \draw[thick] (1,3) -- (1.75,3.75);                 
    \draw[dashed] (0,3.75) -- (4.8,3.75);              
    \node at (-0.6,3.75) {$k+2$};                      
    \filldraw (1.75,3.75) circle (2pt);                
    \draw[fill=white] (3.25,3.75) circle (2pt);        
    \draw[fill=white] (4,3.75) circle (2pt);           
    \draw[thick] (1.75,3.75) -- (4,4.5);               
    \draw[dashed] (0,4.5) -- (4.8,4.5);                
    \filldraw (4,4.5) circle (2pt);                    
  \end{tikzpicture}
  \caption{The path through points in $\mathcal{T}_k$. Each black dot represents $(g(x_k),\inner{b}{x_k})$ that corresponds to $x_k \in \mathcal{T}_k$, while the white dots correspond to other feasible points. The red thick line is the path through the Pareto optimal points.}
  \label{fig:border_points}
  \end{center}
\end{figure}

The main idea of the algorithm is to generate the sequence $\{x_k\}$ by applying minimum transitions at each iteration. However, it is not necessary to start this procedure from $x_{\bar{k}}$. In fact, $x_j \in \mathcal T_j$ as in Theorem~\ref{theo:Pareto} can be characterized as the lexicographically optimal solution under the preference order of $g$ over $b$. Gorski et al.~\cite{GKS23} made a similar statement for biobjective linear optimization on matroids. 

Thus, the algorithm starts with finding a global minimizer of $g$. This can be done with polynomial-time algorithms for M$^\natural$-convex function minimization \cite{Shio04,Tam05}. Let $R(x)$ denote the set of transitions with respect to $x\in\dom g$. The algorithm selects a transition $(u,v)$ that minimizes $g(x;u,v)$.
If $g(x;u,v)=0$, then $x$ is dominated by $x-\chi_u+\chi_v$, which means that $x$ is not a Pareto optimal solution. If $g(x;u,v)>0$, then $x$ is a Pareto optimal solution. In either case, the algorithm updates $x$ to $x-\chi_u+\chi_v$. This decreases $\langle b,x\rangle$ by one. The algorithm repeats this process until $R(x)$ becomes empty. 



Based on the above theoretical background, we present the overview of the method in Algorithm~\ref{alg:MnatBB}. Let $T_{\mathrm{M}^\natural}$ denote the time required for a single evaluation of $g$. The M$^\natural$-convex function minimization can be done in 
$\mathcal O(|E|^3T_{\mathrm{M}^\natural} \log\frac{L}{|E|})$ time \cite{Shio04,Tam05}, 
where 
$$L:=\max\{\|x-y\|_\infty\mid x,y\in\dom g\}.$$ 
The number of iterations of the while-loop w.r.t. $k$ is $\mathcal{O}(K)$, where $K:=\overline{k}-\underline{k}$. Moreover, the line~4 requires $\mathcal{O}(|E|^2)$ evaluations of~$g$. Thus, the computational time of the iterations is $\mathcal{O}(|E|^2KT_{\mathrm{M}^\natural})$. Since \red{$K=\mathcal{O}(|E|L)$, 
the overall running time is $\mathcal{O}(|E|^3LT_{\mathrm{M}^\natural})$.} 

In the special cases where $\dom g \subseteq \{0,1\}^E$, \red{the number of iterations $K$ can be bounded by $r:=\max\{\|x\|_1\mid x\in\dom g\}$.}  
The initial minimization of $g$ can be done in \red{$O(|E|^2T_{\mathrm{M}^\natural})$} time, which is dominated by the iterations. Thus, the algorithm runs in \red{$\mathcal{O}(r|E|^2T_{\mathrm{M}^\natural})$} time in the $0$-$1$ cases.

\begin{algorithm}[H]
	\caption{M$^\natural$BB algorithm}
	\label{alg:MnatBB}
	\begin{algorithmic}[1]
	\REQUIRE An M$^\natural$ convex function $g$, and a binary linear function $\inner{b}{\cdot}$. 
	\ENSURE Complete set $\mathcal X_{\mathrm{cP}}$ of Pareto optimal values. 
	\STATE Find an $x\in \dom g$ that minimizes $g(x)$. 
        \STATE $k\gets\langle b,x \rangle$, $\mathcal  X_{\mathrm{cP}}\leftarrow \varnothing$.  
    \WHILE{$R(x)\neq\varnothing$}
    \STATE Find a transition $(u,v)\in R(x)$ 
    that minimizes $g(x;u,v)$. 
    \IF{$g(x;u,v)>0$}
    \STATE $\mathcal X_{\mathrm{cP}} \leftarrow \mathcal X_{\mathrm{cP}} \cup \{(g(x),k)\}$
    \ENDIF
    \STATE $x \leftarrow x -\chi_u+\chi_v$
	\STATE $k \leftarrow k-1$
	\ENDWHILE
    \STATE $\mathcal X_{\mathrm{cP}} \leftarrow \mathcal X_{\mathrm{cP}} \cup \{(g(x),k)\}$
    \RETURN $\mathcal X_{\mathrm{cP}}$
    \end{algorithmic}
\end{algorithm}


\subsection*{Application to branching with arc and root costs}
\label{sec:AND}
We now provide an example of biobjective optimization with M$^\natural$-convex and binary linear functions.

Let $G=(V,A)$ be a directed graph with a vertex set $V$ and an arc set $A$. Consider a source set $S\subseteq V$, an arc cost function $c\colon A\to\Re_+$, and a source cost function $b\colon S\to\{0,1\}$. We intend to select a pair of subsets $F\subseteq A$ and $R\subseteq S$ so that each vertex $v\in V$ is reachable from 
some source in $R$ 
in the subgraph $H=(V,F)$. How can we minimize the costs of this selection measured by $c(F):=\sum_{a\in F}c(a)$ and $b(R):=\sum_{s\in R}b(s)$?  
This problem can be formulated as a biobjective optimization problem in our framework as follows.

A subset $F\subseteq A$ is called a branching if $F$ contains no undirected cycle and at most one arc entering $v$ for each vertex $v\in V$. The root of a branching $F$ denoted by $R(F)$ is the set of vertices that is not entered by any arc in $F$. Due to the nonnegativity of the arc costs, the biobjective optimization problem above is tantamount to finding a branching $F$ with $R(F)\subseteq S$ minimizing $c(F)$ and $b(R(F))$ simultaneously. An example of this problem is shown in Fig.~\ref{fig:example-M-nat}.

For a specified source subset $R\subseteq S$, let $\gamma(R)$ denote 
the minimum cost $c(F)$ among all branchings $F\subseteq A$ with $R(F)=R$. 
We now consider a function $g\colon\Z^S\to\overline{\Re}$ defined by 
$g(x):=\gamma(R)$ for $x=\chi_R$ and $g(x):=\infty$ otherwise. 
Takazawa~\cite{Tak12,Tak14} observed that the function $g$ thus defined is an M$^\natural$-convex function and utilized this property for investigating the optimal matching forest and shortest bibranching problems.

Given a terminal subset $R\subseteq S$, one can compute $\gamma(R)$ by shrinking $R$ into a single vertex $r$ and finding a shortest $r$-arborescence in the resulting directed graph. This can be done in $\mathcal{O}(|A|+|V|\log |V|)$ time \cite{GGST86}. Therefore, our algorithm enumerates all the Pareto optimal values in $\mathcal{O}(|S|^3(|A|+|V|\log |V|))$ time in this setting. 

\begin{figure}[t!]
  \centering
  \begin{tabular}{@{}c@{\hspace{5pt}}c@{\hspace{5pt}}c@{}}
  
    \begin{tikzpicture}[
      >=Stealth,
      vertex/.style={circle, draw, thick, minimum size=12pt, inner sep=0pt, font=\small},
      edge/.style={->, thick},
      weight/.style={font=\small}
      ]
      \node[vertex, fill=red!30] (4) at (-0.9,0.75) {4};
      \node[vertex] (5) at (-0.9,-0.75) {5};
      \node[vertex] (1) at (0,0) {1};
      \node[vertex, fill=red!30] (2) at (1.5,0) {2};
      \node[vertex] (3) at (2.4,0.75) {3};
      \node[vertex, fill=red!30] (6) at (2.4,-0.75) {6};
      \draw[edge] (1) -- node[weight, pos=0.35, above] {1} (4);
      \draw[edge] (4) -- node[weight, pos=0.4, left] {1} (5);
      \draw[edge] (1) -- node[weight, pos=0.35, below] {5} (5);       
      \draw[edge, bend left=25] (1) to node[weight, above] {2} (2);
      \draw[edge, bend left=25] (2) to node[weight, below] {3} (1);
      \draw[edge] (2) -- node[weight, pos=0.35, above] {5} (3);
      \draw[edge] (2) -- node[weight, pos=0.35, below] {1} (6);
      \draw[edge] (6) -- node[weight, pos=0.4, right] {1} (3);
    \end{tikzpicture}
    
    &    
    &
    
    \begin{tikzpicture}[
      >=Stealth,
      vertex/.style={circle, draw, thick, minimum size=12pt, inner sep=0pt, font=\small},
      edge/.style={->, thick},
      weight/.style={font=\small}
      ]   
      \node[vertex] (4) at (-0.9,0.75) {4};
      \node[vertex] (5) at (-0.9,-0.75) {5};
      \node[vertex] (1) at (0,0) {1};
      \node[vertex, fill=blue!20] (2) at (1.5,0) {2};
      \node[vertex] (3) at (2.4,0.75) {3};
      \node[vertex] (6) at (2.4,-0.75) {6};      
      \draw[edge] (1) -- node[weight, pos=0.35, above] {1} (4);
      \draw[edge] (4) -- node[weight, pos=0.4, left] {1} (5);
      \draw[edge, bend left=25] (2) to node[weight, below] {3} (1);
      \draw[edge] (2) -- node[weight, pos=0.35, below] {1} (6);
      \draw[edge] (6) -- node[weight, pos=0.4, right] {1} (3);
    \end{tikzpicture}

    \\
    
    \begin{tabular}{c} 
      Original graph: \\
      $b(2)=1$, \\ 
      $b(4)=0$, $b(6)=1$ 
    \end{tabular} 
    & 
    \begin{tabular}{c}
      Cases $R=\{4\}, \{6\}$ \\ and $\{4,6\}$:\\ 
      Infeasible
    \end{tabular} 
    &
    \begin{tabular}{c} 
      Case $R=\{2\}$: \\ $c(F)=7$, $b(R(F))=1$ \\ Not Pareto optimal
    \end{tabular}

    \\[20pt]

    \begin{tikzpicture}[
      >=Stealth,
      vertex/.style={circle, draw, thick, minimum size=12pt, inner sep=0pt, font=\small},
      edge/.style={->, thick},
      weight/.style={font=\small}
      ]   
      \node[vertex, fill=blue!20] (4) at (-0.9,0.75) {4};
      \node[vertex] (5) at (-0.9,-0.75) {5};
      \node[vertex] (1) at (0,0) {1};
      \node[vertex, fill=blue!20] (2) at (1.5,0) {2};
      \node[vertex] (3) at (2.4,0.75) {3};
      \node[vertex] (6) at (2.4,-0.75) {6};   
      \draw[edge] (4) -- node[weight, pos=0.4, left] {1} (5);
      \draw[edge, bend left=25] (2) to node[weight, below] {3} (1);
      \draw[edge] (2) -- node[weight, pos=0.35, below] {1} (6);
      \draw[edge] (6) -- node[weight, pos=0.4, right] {1} (3);
    \end{tikzpicture}

    &
    
    \begin{tikzpicture}[
      >=Stealth,
      vertex/.style={circle, draw, thick, minimum size=12pt, inner sep=0pt, font=\small},
      edge/.style={->, thick},
      weight/.style={font=\small}
      ]   
      \node[vertex] (4) at (-0.9,0.75) {4};
      \node[vertex] (5) at (-0.9,-0.75) {5};
      \node[vertex] (1) at (0,0) {1};
      \node[vertex, fill=blue!20] (2) at (1.5,0) {2};
      \node[vertex] (3) at (2.4,0.75) {3};
      \node[vertex, fill=blue!20] (6) at (2.4,-0.75) {6};
      \draw[edge] (1) -- node[weight, pos=0.35, above] {1} (4);
      \draw[edge] (4) -- node[weight, pos=0.4, left] {1} (5);
      \draw[edge, bend left=25] (2) to node[weight, below] {3} (1);
      \draw[edge] (6) -- node[weight, pos=0.4, right] {1} (3);
    \end{tikzpicture}

    &

    \begin{tikzpicture}[
      >=Stealth,
      vertex/.style={circle, draw, thick, minimum size=12pt, inner sep=0pt, font=\small},
      edge/.style={->, thick},
      weight/.style={font=\small}
      ]         
      \node[vertex, fill=blue!20] (4) at (-0.9,0.75) {4};
      \node[vertex] (5) at (-0.9,-0.75) {5};
      \node[vertex] (1) at (0,0) {1};
      \node[vertex, fill=blue!20] (2) at (1.5,0) {2};
      \node[vertex] (3) at (2.4,0.75) {3};
      \node[vertex, fill=blue!20] (6) at (2.4,-0.75) {6};
      \draw[edge] (4) -- node[weight, pos=0.4, left] {1} (5);
      \draw[edge, bend left=25] (2) to node[weight, below] {3} (1);
      \draw[edge] (6) -- node[weight, pos=0.4, right] {1} (3);
    \end{tikzpicture}

    \\

    \begin{tabular}{c} 
      Case $R=\{2,4\}$: \\
      $c(F)=6$, $b(R(F))=1$ \\
      Pareto optimal
    \end{tabular} &
    \begin{tabular}{c} 
      Case $R=\{2,6\}$: \\
      $c(F)=6$, $b(R(F))=2$ \\
      Not Pareto optimal
    \end{tabular} &
    \begin{tabular}{c} 
      Case $R=\{2,4,6\}$: \\
      $c(F)=5$, $b(R(F))=2$ \\
      Pareto optimal
    \end{tabular} 
  \end{tabular}
  \caption{\blue{Example of a branching problem with arc and root costs. We consider $V = \{1,2,3,4,5,6\}$, and the arc costs are indicated on the corresponding arcs of the graph. The subset of vertices $S=\{2,4,6\}$ is indicated in red.  We illustrate the minimal-cost branchings for each~$R \subseteq S$, shown in blue. In this case, we have two Pareto optimal solutions.}}
  \label{fig:example-M-nat}
\end{figure}


\section{Biobjective optimization with M-convex and binary linear functions}
\label{sec:MBB}

Let $h:\Z^E\to\overline{\Re}$ be an M-convex function and 
$b:E\to\{0,1\}$ be a binary function. Recall that 
$\langle b,x\rangle:=\sum_{e\in E}b(e)x(e)$ for any $x\in\Z^E$. 
In this section, we consider the following biobjective problem:
\begin{align}
  \min \quad & (h(x),\langle b,x\rangle) \label{prob:MBB} \tag{MBB}\\
  \mbox{s.t.} \quad &x \in \dom h. \nonumber 
\end{align}
Since all M-convex functions are M$^\natural$-convex, this problem is a particular case of~\eqref{prob:MnatBB}. Yet, it is more general than~\eqref{prob:LBB}. We will show that \eqref{prob:MBB} admits a more efficient algorithm than the one presented in 
Section~\ref{sec:MnatBB} for general \eqref{prob:MnatBB}.

Similarly to the previous section, we consider the following partition of $\Q:=\dom h$: 
\[
  \Q_i := \{ x \in \Q \mid \langle b,x\rangle = i\}.
\]
The set of optimal solutions with respect to $h$ is defined by
\[
  \mathcal S_i := \{ x \in \Q_i \mid  h(x) \leq h(y) 
  \mbox{ for all } y \in \mathcal Q_i \},
\]
which corresponds to the set of optimal solutions of the following subproblem:

\begin{align}
  \min \quad & h(x) \label{prob:MBBi} \tag{MBB$_{=i}$} \\
  \mbox{s.t.} \quad & x \in \Q, \nonumber \\
  & \langle b,x \rangle = i. \nonumber
\end{align}
We also define the indices below:
\[
  \underline{k} := \min \{ k \mid \mathcal S_k \neq \varnothing\} \quad
  \mbox{and} \quad \bar k := \max \{ k \mid \mathcal S_k \neq \varnothing\}.
\]
The transition used here is the same as in Definition~\ref{def:transition}, with $\mathcal{P}$ replaced by $\mathcal{Q}$. However, in the case of M-convex functions, the dummy symbol~$z$ is no longer necessary. The following proposition is an extension of \cite[Corollary 3.1]{GT84} from Gabow and Tarjan to the case of M-convex functions.

\begin{proposition}
  \label{prop:MBB}
  For any $x\in \mathcal S_i$ and $x'\in \mathcal S_{\underline k}$ with $i \in \{\underline k+1,\ldots,\bar k\}$ and $\underline k+1 \leq \bar k$, there exists a minimum transition $(u,v)$ with respect to $x$ such that $u\in\supp(x-x')$ and $v\in\supp(x'-x)$.
\end{proposition}

\begin{proof}
  Let $(s,t)$ be a minimum transition with respect to $x$. From the definition of transition, in particular, we have
  \begin{equation}
    \label{eq:transition_B}
    s \in E_1, \quad t \in E_0,  
    \quad \mbox{and} \quad x-\chi_s+\chi_t\in\Q.
  \end{equation}

  \red{We first claim that there exists a minimum transition $(s,v)$ with $v\in \supp(x'-x)$. If $t\in \supp(x'-x)$, we can take $v=t$. Therefore, we may assume $t\notin\supp(x'-x)$, which implies $t \in\supp(x-\chi_s+\chi_t-x')$.}
%
%
%
From the M-convexity of $h$ and the last expression in \eqref{eq:transition_B}, there exists \red{$v \in\supp(x'-x+\chi_s-\chi_t)$} such that $x-\chi_s+\chi_v, \: x'-\chi_v+\chi_t\in\Q$, and
  \begin{equation}
    \label{eq:M-convex-h}
    h(x-\chi_s+\chi_t) + h(x') \geq h(x-\chi_s+\chi_v) + h(x'-\chi_v+\chi_t).      
  \end{equation}
  From~\eqref{eq:transition_B}, we obtain $b(t) = 0$. If $b(v)=1$, we get $x'-\chi_v+\chi_t\in \Q_{\underline k-1}$, which contradicts the fact that $\Q_{\underline k -1} = \varnothing$. Thus, we must have $b(v)=0$, which implies $v \in E_0$. On the other hand, since $b(s)=1$, we obtain $v\neq s$, \red{and hence $v\in\supp(x'-x)$.} Since $x-\chi_s+\chi_v\in \Q$, we conclude that $(s,v)$ is a transition with respect to $x$.

  From $x' \in \mathcal S_{\underline k}$ and $x' -\chi_v +\chi_t\in \Q_{\underline k}$ we have $h(x')\leq h(x' -\chi_v+\chi_t)$, which together with~\eqref{eq:M-convex-h} \red{implies}
  \[
    h(x-\chi_s+\chi_t) \geq h(x-\chi_s+\chi_v), 
  \]
  and \red{hence}
  \[
    h(x;s,t) = h(x-\chi_s+\chi_t)-h(x) \geq h(x-\chi_s+\chi_v) - h(x) = h(x;s,v).
  \]
  Since $(s,t)$ is taken as a minimum transition, the above inequality implies that $(s,v)$ is also a minimum transition, \red{which completes the proof of the claim.} 

  \red{We now prove the existence of a minimum transition $(u,v)$ with $u \in \supp(x-x')$ and $v \in\supp(x'-x)$.} 
  If $s\in\supp(x-x')$, then we can take $u = s$, so we may assume that $s\notin\supp(x-x')$. This also implies that $s\in\supp(x'-x+\chi_s-\chi_v)$. Once again from the M-convexity of $h$, there exists an element \red{$u\in\supp(x -\chi_s+\chi_v-x')$} such that $x-\chi_u+\chi_v$, $x'-\chi_s+\chi_u \in \Q$, and 
  \begin{equation}
    \label{eq:eq:M-convex-h2}
    h(x-\chi_s+\chi_v) +h(x') \geq h(x-\chi_u+\chi_v)+h(x'-\chi_s+\chi_u).      
  \end{equation} 
  Recalling~\eqref{eq:transition_B}, if $b(u) = 0$, we have $x'-\chi_s+\chi_u\in \Q_{\underline k-1}$, which contradicts the fact that $\Q_{\underline k -1} = \varnothing$. Thus, $b(u)=1$ holds and consequently $u \in E_1$, \red{which implies $u\neq v$, and hence $u\in\supp(x-x')$.} From the previous discussion, \red{we have already seen} $v \in E_0$. \red{Since} $x-\chi_u+\chi_v \in \Q$ also holds, we \red{may assert} that $(u,v)$ is a transition with respect to $x$. 
  
  \red{Since} $x' \in \mathcal S_{\underline k}$ and $x' -\chi_s +\chi_u \in \Q_{\underline k}$, we \red{have} $h(x' -\chi_s +\chi_u) \geq h(x')$. This, together with~\eqref{eq:eq:M-convex-h2}, \red{implies}
  \[
    h(x-\chi_s+\chi_v) \geq h(x-\chi_u+\chi_v)
  \]
  and \red{hence} 
  \[
    h(x;s,v) = h(x-\chi_s+\chi_v)-h(x) \geq h(x-\chi_u+\chi_v) - h(x) = h(x;u,v).
  \]
  Since $(s,v)$ is a minimum transition, we conclude from the above result that $(u,v)$ is also a minimum transition with respect to $x$.
%
\end{proof}

The above result shows how we can solve \eqref{prob:MBB}.  
The procedure is given in Algorithm~\ref{alg:MBB}. We first compute the initial $x_k$ and the target $x'$ by solving two minimization problems with \mbox{M-convex} functions. The main difference from Algorithm~\ref{alg:MnatBB} lies in the selection of the minimum transitions $(u_k,v_k)$ that are always chosen from the restricted set $\Delta_1 \times \Delta_0$ with $\Delta_1:=E_1\cap\supp(x_k-x')$ and $\Delta_0:=E_0\cap\supp(x'-x_k)$. Once a minimum transition is found, the algorithm updates~$x$ to~$x-\chi_u+\chi_v$, and the algorithm repeats until the iteration $\underline{k}-1$. Once again, the number of iterations is bounded by $K:=\overline k -\underline k$. In each iteration, the transition cost is computed for every pair of elements taken from the sets $\Delta_1$ and $\Delta_0$. The iterations take $\mathcal{O}(|E|^2KT_{\mathrm{M}})$ time, which \red{dominates} the executions of M-convex function minimization at the beginning. Thus, the asymptotic running time bound of Algorithm~\ref{alg:MBB} for \eqref{prob:MBB} is the same as that of Algorithm~\ref{alg:MnatBB}.  

We now consider a special case where the effective domain of $h$ is restricted to $0$-$1$ vectors. In this case, all feasible points correspond to characteristic vectors of bases of a matroid~$\M$. Let $r$ denote the rank of $\M$. The minimization of an M-convex function on a matroid can be performed by the greedy algorithm in $\mathcal O(r|E|T_\mathrm{M})$ time. Minimizing the linear function $\langle b,\cdot\rangle$ over the bases of $\M$ can also be done by the greedy algorithm in the same running time bound. Note that the set $\Q'$ of minimizers forms a base family of a matroid. Let $h'$ denote the function defined by
$h'(x)=h(x)$ for $x\in\Q'$ and $h'(x)=+\infty$ for $x\notin \Q'$. Then 
$h'$ is again an M-convex function, and hence one can minimize $h'$ in $\mathcal O(r|E|T_\mathrm{M})$ time. We also have $K\leq r$, $|\Delta_0|\leq r$, and $|\Delta_1|\leq r$, which imply that the total computational cost of the iterations is $\mathcal O(r^3 T_\mathrm{M})$. Thus, the entire running time of Algorithm~\ref{alg:MBB} is $\mathcal O(r(r^2+|E|)T_\mathrm{M})$.
This improves over the \red{$\mathcal{O}(r|E|^2T_\mathrm{M})$} running time bound of Algorithm~\ref{alg:MnatBB} directly applied to the problems with $0$-$1$ effective domains.

\X{
Let us now analyze the complexity of Algorithm~\ref{alg:MBB}. 
Define $r = \mathrm{rank}(\mathcal M)$. When decreasing the binary value $k$ from $j$ to $\underline{k}$, we have 
$\{\underline k,\ldots,j\} \subseteq \{ 0,\ldots,r\}$. Thus, the size of the while-loop with respect to $k$ is $\mathcal{O}(r)$.
Also, in each iteration, the transition cost is computed for every pair of elements taken from the sets $\Delta_1$ and $\Delta_0$. Since the minimum transition used in each step is removed from consideration, the number of iterations of the for-loop is bounded above by $(r - (j - k))^2$. If $T_{\mathrm{M}}$ is the time required to evaluate the M-convex function, then the total computational time over the entire while-loop is bounded above by $\sum_{k= \underline k}^j (r-(j-k))^2\, T_{\mathrm{M}}$, which can be evaluated as follows:
\begin{align*}
  \sum_{k= \underline k}^j (r-(j-k))^2\, T_{\mathrm{M}} 
  & \leq \sum_{k=0}^r (r-(r-k))^2 \,T_M \\
  & =\frac16 r(r+1)(2r+1) T_M\\
  & = \mathcal O(r^3 T_M).
\end{align*}    
Therefore, the total computational time of Algorithm~\ref{alg:MBB} in this case is given by
\[
  2T_{\mathrm{min\,M}} + \mathcal O(r^3T_{\mathrm{M}} ), 
\]
where $T_{\mathrm{min\,M}}$ is the time required to minimize an M-convex function on a matroid. 

Note that if only Lemma~\ref{lem:MBB} is considered, it will be necessary to compute the set difference between the bases at each step and to compare all the values from scratch. In contrast, Algorithm~\ref{alg:MBB} uses the result of Theorem~\ref{theo:MBB}, which reduces the search space linearly, consequently making the method more efficient.
}

\begin{algorithm}[H]
	\caption{MBB algorithm}
	\label{alg:MBB}
	\begin{algorithmic}[1] 
	\REQUIRE An M-convex function~$h$, and a binary linear function $\inner{b}{\cdot}$.
	\ENSURE Complete set of Pareto optimal values $\mathcal X_{\mathrm{cP}}$
    \STATE Find an $x\in\dom h$ that minimizes $h(x)$. 
    \STATE $k\gets\langle b,x \rangle$, $\mathcal X_{\mathrm{cP}} \gets \varnothing$
    \STATE Find an $x'\in\Z^E$ that minimizes $h(x')$ 
           in $\Arg\min_{y} \{\langle b,y\rangle\mid y\in\dom h\}$.
	\STATE $\underline k\gets\langle b,x'\rangle$ 
	\WHILE{$k>\underline{k}$}
        \STATE $\Delta_1 \leftarrow E_1\cap \supp(x-x'),\ \Delta_0 \leftarrow E_0\cap \supp(x'-x)$
        \STATE Find a transition $(u,v) \in \Delta_1 \times \Delta_0$ that minimizes $h(x;u,v)$. 
        \IF{$h(x;u,v) > 0$}
            \STATE $\mathcal X_{\mathrm{cP}} \leftarrow 
                    \mathcal X_{\mathrm{cP}} \cup \{(h(x),k)\}$
        \ENDIF
        \STATE $x \leftarrow x -\chi_u+\chi_v$
        \STATE $k \leftarrow k-1$
	\ENDWHILE
    \RETURN $\mathcal X_{\mathrm{cP}}$
	\end{algorithmic}
\end{algorithm}


\subsection*{Application to bipartite matching with edge and vertex costs}

We now provide an example of biobjective optimization with M-convex and binary linear functions.

Let $H=(S,T;A)$ be a bipartite graph with disjoint vertex sets $S,T$ and edge set $A$, where each edge in $A$ is between $S$ and $T$. An edge subset $M\subseteq A$ is called a matching in $H$ if no two distinct edges in $M$ share an end-vertex. In other words, a matching $M$ is a subset of $A$ such that $|\partial M|=2|M|$, where $\partial M$ designates the set of end-vertices of the edges in $M$.

{A matching $M$ is called $T$-perfect if it satisfies $\partial M\supseteq T$. Note that a $T$-perfect matching represents an assignment, i.e., one-to-one mapping from $T$ to $S$. Suppose that a bipartite graph $H$ has a $T$-perfect matching. Let $\B$ denote the family of subsets $R\subseteq S$ that admit a matching $M$ with $\partial M=R\cup T$. 
Then $\B$ is known to satisfy the simultaneous exchange property~\textbf{(B)}, and $(S,\B)$ is called a transversal matroid.

Consider a cost function $c \colon A\to\Re$ and a binary function $b \colon S\to\{0,1\}$.
Our aim is to find a pair of a vertex subset $R\subseteq S$ and a $T$-perfect matching $M$ with $\partial M=R\cup T$ so that $b(R):=\sum_{s\in R}b(s)$ and $c(M):=\sum_{e\in M}c(e)$ are minimized simultaneously. See Fig.~\ref{fig:example_M} for a simple example. For a base $R\in\B$ of the transversal matroid defined by $H$, let $\gamma(R)$ denote the minimum cost $c(M)$ among all the $T$-perfect matchings $M$ with $\partial M=R\cup T$. A function $h\colon\Z^S\to\overline{\Re}$ is now defined by $h(x):=\gamma(R)$ if $x=\chi_R$ for some $R\in\B$ and $h(x):=\infty$ otherwise. The function $h$ thus defined is known to be an M-convex function.

Given a subset $R\subseteq S$ with $|R|=|T|$ one can compute the function value $h(\chi_R)$ in $\mathcal{O}(|T|(|A|+|S|\log|S|))$ time. Therefore, a direct application of Algorithm \ref{alg:MBB} enumerates all the Pareto optimal values in $\mathcal{O}(|T|^2(|T|^2+|S|)(|A|+|S|\log|S|))$ time. Reducing the computation of minimum transition to a shortest path problem, one can improve this algorithm to run in $\mathcal{O}(|T|(|A|+|S|\log|S|))$ time as follows.

\begin{figure}[t!]
  \centering
  \begin{tabular}{@{}c@{\hspace{5pt}}c@{\hspace{5pt}}c@{\hspace{5pt}}c@{\hspace{5pt}}c@{}}
  \begin{tikzpicture}[
    node/.style={circle,draw,minimum size=12pt,inner sep=0pt,font=\small},
    edge/.style={thick},
    cost/.style={font=\small}
    ]
    \node[node] (1) at (0,1) {1};
    \node[node] (2) at (0,0) {2};
    \node[node] (3) at (0,-1) {3};
    \node[node] (4) at (1.75,0.5) {4};
    \node[node] (5) at (1.75,-0.5) {5};
    \draw[edge] (1) -- node[cost, pos=0.4, above] {3} (4);
    \draw[edge] (2) -- node[cost, pos=0.4, above] {1} (4);
    \draw[edge] (2) -- node[cost, pos=0.7, yshift=4pt, right] {2} (5);
    \draw[edge] (3) -- node[cost, pos=0.25, yshift=2pt, left] {2} (4);
    \draw[edge] (3) -- node[cost, pos=0.4, below] {1} (5);
  \end{tikzpicture}
  &
  \begin{tikzpicture}[
    node/.style={circle,draw,minimum size=12pt,inner sep=0pt,font=\small},
    edge/.style={thick},
    cost/.style={font=\small}
    ]
    \node[node] (1) at (0,1) {1};
    \node[node] (2) at (0,0) {2};
    \node[node] (3) at (0,-1) {3};
    \node[node] (4) at (1.75,0.5) {4};
    \node[node] (5) at (1.75,-0.5) {5};
    \draw[edge] (1) -- node[cost, pos=0.4, above] {3} (4);
    \draw[edge, red, very thick] (2) -- node[cost, pos=0.4, above] {1} (4);
    \draw[edge] (2) -- node[cost, pos=0.7, yshift=4pt, right] {2} (5);
    \draw[edge] (3) -- node[cost, pos=0.25, yshift=2pt, left] {2} (4);
    \draw[edge, red, very thick] (3) -- node[cost, pos=0.4, below] {1} (5);
  \end{tikzpicture}
  &
  \begin{tikzpicture}[
    node/.style={circle,draw,minimum size=12pt,inner sep=0pt,font=\small},
    edge/.style={thick},
    cost/.style={font=\small}
    ]
    \node[node] (1) at (0,1) {1};
    \node[node] (2) at (0,0) {2};
    \node[node] (3) at (0,-1) {3};
    \node[node] (4) at (1.75,0.5) {4};
    \node[node] (5) at (1.75,-0.5) {5};
    \draw[edge, red, very thick] (1) -- node[cost, pos=0.4, above] {3} (4);
    \draw[edge] (2) -- node[cost, pos=0.4, above] {1} (4);
    \draw[edge, red, very thick] (2) -- node[cost, pos=0.7, yshift=4pt, right] {2} (5);
    \draw[edge] (3) -- node[cost, pos=0.25, yshift=2pt, left] {2} (4);
    \draw[edge] (3) -- node[cost, pos=0.4, below] {1} (5);
  \end{tikzpicture}
  &
  \begin{tikzpicture}[
    node/.style={circle,draw,minimum size=12pt,inner sep=0pt,font=\small},
    edge/.style={thick},
    cost/.style={font=\small}
    ]
    \node[node] (1) at (0,1) {1};
    \node[node] (2) at (0,0) {2};
    \node[node] (3) at (0,-1) {3};
    \node[node] (4) at (1.75,0.5) {4};
    \node[node] (5) at (1.75,-0.5) {5};
    \draw[edge] (1) -- node[cost, pos=0.4, above] {3} (4);
    \draw[edge] (2) -- node[cost, pos=0.4, above] {1} (4);
    \draw[edge, red, very thick] (2) -- node[cost, pos=0.7, yshift=4pt, right] {2} (5);
    \draw[edge, red, very thick] (3) -- node[cost, pos=0.25, yshift=2pt, left] {2} (4);
    \draw[edge] (3) -- node[cost, pos=0.4, below] {1} (5);
  \end{tikzpicture}
  &
  \begin{tikzpicture}[
    node/.style={circle,draw,minimum size=12pt,inner sep=0pt,font=\small},
    edge/.style={thick},
    cost/.style={font=\small}
    ]
    \node[node] (1) at (0,1) {1};
    \node[node] (2) at (0,0) {2};
    \node[node] (3) at (0,-1) {3};
    \node[node] (4) at (1.75,0.5) {4};
    \node[node] (5) at (1.75,-0.5) {5};
    \draw[edge, red, very thick] (1) -- node[cost, pos=0.4, above] {3} (4);
    \draw[edge] (2) -- node[cost, pos=0.4, above] {1} (4);
    \draw[edge] (2) -- node[cost, pos=0.7, yshift=4pt, right] {2} (5);
    \draw[edge] (3) -- node[cost, pos=0.25, yshift=2pt, left] {2} (4);
    \draw[edge, red, very thick] (3) -- node[cost, pos=0.4, below] {1} (5);
  \end{tikzpicture}
  \\
  \begin{tabular}{@{}c@{}} $b(1)=0$ \\ $b(2)=1$ \\ $b(3)=1$ \end{tabular} &
  \begin{tabular}{c} 
    \begin{tabular}{r} $c(M)=2$ \\ $b(R)=2$ \end{tabular} \\ 
    Pareto opt. \end{tabular} &
  \begin{tabular}{@{}c@{}} 
    \begin{tabular}{r} $c(M)=5$ \\ $b(R)=1$ \end{tabular} \\ 
    Not Pareto opt. \end{tabular} &
  \begin{tabular}{@{}c@{}} 
    \begin{tabular}{r} $c(M)=4$ \\ $b(R)=2$ \end{tabular} \\ 
    Not Pareto opt. \end{tabular} &
  \begin{tabular}{@{}c@{}} 
    \begin{tabular}{r} $c(M)=4$ \\ $b(R)=1$ \end{tabular} \\ 
    Pareto opt. \end{tabular} \\
  \end{tabular}
  \caption{\blue{Example of a bipartite matching problem with edge and vertex costs. We consider $S = \{1,2,3\}$, $T=\{4,5\}$, and the edge costs are indicated next to the corresponding edges. In this case, we have four possible $T$-perfect matchings, shown as thick red lines, and two of them are Pareto optimal. Considering the improved algorithm, note that given a perfect matching that minimizes $c(M)$ -- in this case, the leftmost matching above -- the shortest path (visiting vertices 1, 4, and 2 in this order) and the corresponding update of the matching yield the rightmost matching, which indeed has $b(R)$ decreased by one.}}
  \label{fig:example_M}
\end{figure}

When finding a $T$-perfect matching $M$ that minimizes $c(M)$, one also obtains an optimal potential $p \colon S\cup T\to\Re$, which satisfies (\rnum{1}) $p(s)+p(t)\leq c(a)$ for any edge $a=(s,t)\in A$ with $s\in S$ and $t\in T$, (\rnum{2}) $p(s)+p(t)=c(a)$ for any edge $a=(s,t)\in M$ with $s\in S$ and $t\in T$, (\rnum{3}) $p(s)\geq 0$ for any $s\in S$, and (\rnum{4}) $p(s)=0$ for any $s\in S\setminus \partial M$. Construct an auxiliary directed graph $D_M=(S\cup T,A)$ in which the arcs in $A\setminus M$ are directed from $S$ to $T$ and those in $M$ are directed from $T$ to $S$. The reduced cost function $c_p \colon A\to\Re$ is defined by $c_p(a):=c(a)-p(s)-p(t)$ for $a=(s,t)\in A\setminus M$ and $c_p(a):=p(s)+p(t)-c(a)$ for $a=(s,t)\in M$, where $s\in S$ and $t\in T$. Note that $c_p(a)\geq 0$ for any $a\in A$, and in particular, $c_p(a)=0$ for $a\in M$. Consider $S_0:=\{s\mid s\in S,\, b(s)=0\}$ and $S_1:=\{s\mid s\in S,\, b(s) =1\}$. One can find a shortest path $P$ in $D_M$ from $S_0\setminus \partial M$ to $S_1\cap \partial M$ with respect to the reduced cost~$c_p$ by Dijkstra's algorithm in $\mathcal{O}(|A|+|S|\log |S|)$ time. Then replace the matching $M$ by the symmetric difference of $M$ and the edge set of $P$. Update the potential $p$ to $p(s):=p(s)-\min\{d(s),\ell(P)\}$ for $s\in S$ and $p(t):=p(t)+\min\{d(t),\ell(P)\}$ for $t\in T$, where $d(v)$ is the shortest path distance from $S_0\setminus \partial M$ to $v\in S\cup T$ and $\ell(P)$ is the length of the shortest path $P$. Note that these updates keep (\rnum{1})--(\rnum{4}) and decrease $b(R)$ by one, where $R=S\cap\partial M$. Consequently, $c(M)$ attains the minimum cost among all the $T$-perfect matchings with the same value of $b(R)$. If there is no directed path from $S_0\setminus\partial M$ to $S_1\cap \partial M$, then there is no $T$-perfect matching $M$ with a smaller value of $b(R)$. Repeating this iteration $O(|T|)$ times, one can obtain all the Pareto optimal values in $\mathcal{O}(|T|(|A|+|S|\log|S|))$ time.


\section{\texorpdfstring{Biobjective optimization with M$^\natural$-convex and lexicographic functions}{Biobjective optimization with M-natural-convex and lexicographic functions}}
\label{sec:MLB}

In this section, instead of optimization over integer lattice points, let us restrict ourselves to g-matroids. We first introduce a preference order over the ground set $E$ of a g-matroid $\M=(E,\mathcal{P})$. Assume that $E$ is partitioned into $m$ categories $E_1,\ldots, E_m$, i.e., $E_i \cap E_j = \varnothing$ for all $i \ne j$ and $E = \bigcup_{k=1}^{m} E_k$. Let us assume the preference order $E_1 \prec E_2 \prec \cdots \prec E_m$, where $E_i \prec E_j$ indicates that elements of $E_j$ are preferred over elements of $E_i$. For a subset $X\subseteq E$, we define a nonnegative vector $\eta(X) \in \Z^m$ by $\eta(X)_i = |E_i \cap X|$, which counts the number of elements from each category contained in $X$. Here, we refer to $\eta$ as a \emph{lexicographic function}, but in some literature, the associated vector is called \emph{ordinal cost}~\cite{KSS23}. 

Before defining the problem, we first recall that the lexicographic cone in $\Re^m$ is given~by
\[
  K_{\mathrm{lex}}^m = \{ 0 \} \cup 
  \{ z \in \Re^m \mid \mbox{there exists } k \mbox{ s.t. } z_k > 0 \mbox{ and }
  z_i = 0 \mbox{ for all } i < k \}.
\]
Clearly, $K_{\mathrm{lex}}^m$ is convex, pointed but neither closed nor open (see Fig.~\ref{fig:lex_cone} for the case $m=2$). The binary relation induced by~$K_{\mathrm{lex}}^m$, called lexicographic order, is actually a total order~\cite{Luc89}. 
We then use the following notation: $x \le_{\mathrm{lex}} y$ if and only if $y-x \in {K^m_\mathrm{lex}}$. We also write $x\lneq_\mathrm{lex}y$ if $x\le_{\mathrm{lex}} y$ and $x\neq y$.  

\begin{figure}[ht!]
  \centering
  \begin{tabular}{c@{}c@{}c@{}c@{}c@{}c@{}c}
      \begin{tabular}{c}
      \begin{tikzpicture}
        \filldraw (0,0) circle (2pt);             
        \draw[dashed, thick] (0,0) -- (1,0);      
        \draw[dashed, thick] (-0.5,0) -- (0,0);   
        \draw[-{Latex}, thick] (0.9,0) -- (1,0);  
        \draw[dashed, thick] (0,-1) -- (0,0);     
        \draw[dashed, thick] (0,0) -- (0,1);      
        \draw[-{Latex}, thick] (0,0.9) -- (0,1);  
        \node at (1.3,0) {$z_1$};
        \node at (-0.3,0.8) {$z_2$};
      \end{tikzpicture}
    \end{tabular}
    &
    \begin{tabular}{c}
      $\bigcup$
    \end{tabular}
    &
    \begin{tabular}{c}
      \begin{tikzpicture}
        \fill[gray!30, pattern color=gray] (0,0) rectangle (1,1);  
        \fill[gray!30, pattern color=gray] (0,0) rectangle (1,-1); 
        \draw[-{Latex}, thick] (0,0) -- (1,0);    
        \draw[dashed, thick] (-0.5,0) -- (0,0);   
        \draw[dashed, thick] (0,-1) -- (0,0);     
        \draw[dashed, thick] (0,0) -- (0,1);      
        \draw[-{Latex}, thick] (0,1);             
        \node at (1.3,0) {$z_1$};
        \node at (-0.3,0.8) {$z_2$};
      \end{tikzpicture}
    \end{tabular}
    &
    \begin{tabular}{c}
      $\bigcup$
    \end{tabular}
    &
    \begin{tabular}{c}
      \begin{tikzpicture}
        \draw[dashed, thick] (0,0) -- (1,0);      
        \draw[dashed, thick] (-0.5,0) -- (0,0);   
        \draw[-{Latex}, thick] (0.9,0) -- (1,0);  
        \draw[dashed, thick] (0,-1) -- (0,0);     
        \draw[-{Latex}, thick] (0,0) -- (0,1);    
        \draw[fill=white] (0,0) circle (2pt);     
        \node at (1.3,0) {$z_1$};
        \node at (-0.3,0.8) {$z_2$};
      \end{tikzpicture}
    \end{tabular}
    &
    \begin{tabular}{c}
      $=$
    \end{tabular}
    &
    \begin{tabular}{c}
      \begin{tikzpicture}
        \fill[gray!30, pattern color=gray] (0,0) rectangle (1,1);  
        \fill[gray!30, pattern color=gray] (0,0) rectangle (1,-1); 
        \filldraw (0,0) circle (2pt);             
        \draw[-{Latex}, thick] (0,0) -- (1,0);    
        \draw[dashed, thick] (-0.5,0) -- (0,0);   
        \draw[dashed, thick] (0,-1) -- (0,0);     
        \draw[-{Latex}, thick] (0,0) -- (0,1);    
        \node at (1.3,0) {$z_1$};
        \node at (-0.3,0.8) {$z_2$};
      \end{tikzpicture}
    \end{tabular}
  \end{tabular}
  \caption{Lexicographic cone in $\Re^2$}
  \label{fig:lex_cone}
\end{figure}

Based on this, we say that $X$ \emph{lexicographically dominates} $Y$ with respect to $\eta$, when $\eta(X) \lneq_{\mathrm{lex}} \eta(Y)$, i.e., if there exists $k\in \{1,\ldots,m\}$ such that
\[
  \eta(X)_k < \eta(Y)_k \quad \mbox{and} \quad
  \eta(X)_i = \eta(Y)_i \mbox{ for all } i\in \{1,\ldots,k-1\}.
\]
In other words, $X$ is preferred to $Y$ if it contains fewer elements from a lower-priority category. When~$k=1$, we consider that the condition on~$i$ is automatically satisfied. We also say that $X^\ast$ \emph{lexicographically minimizes} $\eta$ among a set $\mathcal{X}$ when $X^\ast$ lexicographically dominates all the points in $\mathcal{X}$ with respect to $\eta$.

We now consider the following biobjective optimization problem:
\begin{align}
  \min_{\Re_+,K^m_\mathrm{lex}} \quad & (g(X),\,\eta(X)) \label{prob:MLB} \tag{MLB} \\
  \mbox{s.t.}\quad & X \in \mathcal P, \nonumber
\end{align}
where $g \colon \mathcal P \to \Re$ is an M$^\natural$-convex function defined in an M$^\natural$-convex family $\mathcal{P}$, and $\eta \colon \mathcal P \to \Z^m$ is a lexicographic function. This minimization is interpreted with respect to different partial orders, i.e., a point $X^* \in \mathcal P$ is Pareto optimal if there is no $X \in \mathcal P$ such that $g(X) \leq g(X^*)$, $\eta(X) \leq_{\mathrm{lex}} \eta(X^*)$, and $(g(X),\eta(X))\neq (g(X^*),\eta(X^*))$. 

Note that this problem generalizes those considered in the previous two sections (except that here the variables are restricted to g-matroids), as well as the one discussed by Klamroth et al.~\cite{KSS23}, where~$g$ is linear with nonnegative integer coefficients. In~\cite{KSS23}, the authors considered an $\varepsilon$-constraint-type method, minimizing $g(X)$ under the constraint $\eta(X) = \varepsilon$, which can further be formulated as a matroid intersection problem. The Pareto optimal value set can then be obtained via this approach by filtering out the dominated points.

Here, we also consider the $\varepsilon$-constraint method, but we keep the ordering of the lexicographic cone. More specifically, we take a nonnegative $\mu \in \Z^m$ and consider the following subproblem:
\begin{align}
  \min \quad & g(X) \label{prob:MLBu} \tag{MLB$_\mu$} \\
  \mbox{s.t.} \quad & X \in \mathcal P, \nonumber\\
  &\eta(X) \lneq_{\mathrm{lex}} \mu. \nonumber
\end{align}
The result below shows that solving the problem~\eqref{prob:MLBu} gives Pareto optimal solutions for the original problem~\eqref{prob:MLB}, without requiring an additional filtering step. 

\X{
\begin{theorem}
\label{th:LexWeakPareto}
  For all nonnegative $\mu \in \Z^m$, the optimal solution of~\eqref{prob:MLBu} is weakly Pareto optimal for~\eqref{prob:MLB}. 
\end{theorem}
\begin{proof}
  Let $X^\ast$ be an optimal solution of~\eqref{prob:MLBu}, and assume that it is not weakly Pareto optimal of~\eqref{prob:MLB}. Then, there exists $X \in \mathcal P$ such that
  \begin{align}
    \label{eq:weakly-Pareto_spec}
    g(X) < g(X^\ast) \quad \mbox{and} \quad \eta(X) \leq_{\mathrm{lex}} \eta(X^\ast).
  \end{align}
  The feasibility of $X^\ast$ for~\eqref{prob:MLBu} yields $\mu - \eta(X^\ast) \in K^m_\mathrm{lex}$, and \eqref{eq:weakly-Pareto_spec} shows that $\eta(X^\ast) - \eta(X)\in K^m_\mathrm{lex}$. Since $K^m_\mathrm{lex}$ is a convex cone, the sum of its two elements remains in~$K^m_\mathrm{lex}$, and thus
  \[
    (\mu - \eta(X^\ast)) + (\eta(X^\ast) - \eta(X)) = \mu - \eta(X) \in K^m_\mathrm{lex},
  \]
  that is, $X$ is feasible to~\eqref{prob:MLBu}. From~\eqref{eq:weakly-Pareto_spec}, it means $X^\ast$ is not optimal to~\eqref{prob:MLBu}, which is a contradiction.
\end{proof}
}

\begin{theorem}
\label{th:LexPareto}
For a nonnegative $\mu \in \Z^m$, suppose that $X^*\in\mathcal{P}$ lexicographically minimizes~$\eta$ among all the optimal solutions of~\eqref{prob:MLBu}. Then $X^*$ is Pareto optimal for \eqref{prob:MLB}.
\end{theorem}
\begin{proof} 
Suppose to the contrary that $X^*$ is not Pareto optimal to \eqref{prob:MLB}. Then there exists $X\in\mathcal{P}$ such that either $g(X)<g(X^*)$ and $\eta(X)\leq_\mathrm{lex}\eta(X^*)$ or $g(X)=g(X^*)$ and $\eta(X)\lneq_\mathrm{lex}\eta(X^*)$ hold. 
Since $X^*$ is an optimal solution of \eqref{prob:MLBu}, there is no $X\in\mathcal{P}$ such that $g(X)<g(X^*)$ and $\eta(X)\leq_\mathrm{lex}\eta(X^*)\lneq_\mathrm{lex}\mu$. Therefore, $X\in\mathcal{P}$ must satisfy $g(X)=g(X^*)$ and $\eta(X)\lneq_\mathrm{lex}\eta(X^*)$, which contradicts the choice of $X^*$. 
\end{proof}

Theorem~\ref{th:LexPareto} suggests the following algorithm for enumerating the Pareto optimal value set of \eqref{prob:MLB}. 
The algorithm starts with finding $X\in\mathcal{P}$ that minimizes $g(X)$ by the greedy algorithm. The set of all minimizers of $g$ forms a g-matroid. Among those minimizers, the algorithm selects one that lexicographically minimizes $\eta$. 
Recall that any $Y\in\mathcal{P}$ with $\eta(X)\lneq_{\mathrm{lex}}\eta(Y)$ is dominated by $X$. 
Thus, $X$ is a Pareto optimal solution, and there is no Pareto optimal solution $Y$ with $\eta(X)\lneq_{\mathrm{lex}}\eta(Y)$.  

Given the current Pareto optimal solution $X$, the algorithm sets $\mu:=\eta(X)$ and solves \eqref{prob:MLBu} to find 
an optimal solution $X^*$ that lexicographically minimizes $\eta$. Theorem~\ref{th:LexPareto} guarantees that $X^*$ is another Pareto optimal solution. In addition, any $Y\in\mathcal{P}$ with $\eta(X^*)\lneq_\mathrm{lex}\eta(Y)\lneq_\mathrm{lex}\eta(X)$ is dominated by $X^*$. The algorithm enumerates all the Pareto optimal values by repeating this process while \eqref{prob:MLBu} is feasible. 

We now discuss how to solve \eqref{prob:MLBu}.
As in Fig.~\ref{fig:lex_cone}, the family $\L$ of subsets $X\subseteq E$ satisfying $\eta(X) \le_{\mathrm{lex}} \mu$ can be written as a disjoint union of the following sets:
$$\L=\bigcup_{k=0}^m \L_k,$$
where 
$$\L_0:=\{X\mid \eta(X)_i=\mu_i \mbox{ for all } i=1,\ldots,m\}$$
and 
$$\L_k:=\{X\mid \eta(X)_i=\mu_i \mbox{ for all } i=1,\ldots,k-1, \mbox{ and } \eta(X)_k < \mu_k \}$$
for $k=1,\ldots,m$.
Also, for any $\lambda, \xi \in \Z^m$ with $0 \le \lambda \le \xi$, we denote 
\[
  \mathcal{I}(\lambda,\xi) := \{ X \subseteq E \mid \lambda_i \le \eta(X)_i \le \xi_i  \mbox{ for all }i=1,\ldots,m\}.
\]
Note that $\mathbf{L}=(E, \mathcal{I}(\lambda,\xi))$ forms a g-matroid.
For each $k=1,\ldots,m$, the family $\L_k$ can be rewritten as $\L_k=\mathcal{I}(\lambda,\xi)$, where~$\lambda$ and~$\xi$ are determined by 
$$
  \lambda_i:=\left\{
    \begin{array}{lc} 
      \mu_i & (i<k), \\ 
      0 & (i\geq k),
    \end{array}
  \right. \quad 
  \xi_i:=\left\{
    \begin{array}{lc} 
      \mu_i & (i<k), \\ 
      \mu_i-1 & (i=k), \\ 
      |E| & (i>k).
    \end{array}
  \right. 
$$
Similarly, we have $\L_0=\mathcal{I}(\mu,\mu)$. 


In order to find an optimal solution to \eqref{prob:MLBu}, 
we solve the optimization subproblem in the form of 
\begin{align}
  \min \quad & g(X) \label{prob:MLBuk} \tag{MLB$_{\mu,k}$}\\
  \mbox{s.t.} \quad & X \in \mathcal{P}\cap\mathcal{L}_k\nonumber
\end{align}
for all $k=1,\ldots,m$, and then select the index $\ell$ that achieves the minimum optimal value. If there is a tie, we adopt the smallest index. Thus we obtain \red{the set $X_\ell$, which} lexicographically minimizes $\eta$ among those optimal solutions $X_k$ of \eqref{prob:MLBuk} for $k=1,\ldots,m$. However, we may still have another solution of (MLB$_{\mu,\ell}$) with lexicographically smaller value of $\eta$ \red{than $X_\ell$}. Therefore, we then identify $X^*$ that minimizes $\eta$ lexicographically among all the optimal solutions of (MLB$_{\mu,\ell}$). This clearly gives an optimal solution of \eqref{prob:MLBu}. In addition, the output $X^*$ lexicographically minimizes~$\eta$ among all the optimal solutions of \eqref{prob:MLBu}.

Obviously, each subproblem~\eqref{prob:MLBuk} can be written as
\begin{align*}
  \min \quad & g(X) \\
  \mbox{s.t.}\quad & X \in \mathcal P \cap \mathcal{I}(\lambda,\xi).
\end{align*}
Defining also the indicator function
\[
  w(X) = \left\{ 
    \begin{array}{ll}
      0,  & \mbox{if } X \in \mathcal{I}(\lambda,\xi),\\
      \infty, & \mbox{otherwise},
    \end{array}
  \right.
\]
we further can write~\eqref{prob:MLBuk} as
\begin{align*}
  \min \quad & g(X)+w(X), \\
  \mbox{s.t.}\quad & X \in \mathcal P. 
\end{align*}
Since $w$ and $g$ are both M$^\natural$-convex, the objective function of the above problem is \mbox{M$^\natural_2$-convex}. The minimization problem of an \mbox{M$^\natural_2$-convex} function is known to be polynomially solvable. Exploiting the g-matroid structure, however, one can reduce this particular problem to the valuated matroid intersection problem introduced by Murota~\cite{Mur96b}. A concrete strongly polynomial algorithm has been presented in \cite{Mur96c}.

The algorithm discussed above for enumerating all the Pareto optimal values of \eqref{prob:MLB} is now described formally in Algorithm~\ref{alg:MLB}. 
The lexicographical minimization of $\eta$ in line~9 can be done as follows. The set of all the optimal solutions of (MLB$_{\mu,\ell}$) forms a g-matroid intersection. We then find a set that minimizes $\eta(\cdot)_{\ell+1}$ by the weighted matroid intersection. The set of all the optimal solutions again forms a g-matroid intersection. We continue this process for $k=\ell+1,\ldots,m$ to obtain the lexicographically optimal solution. The computational cost for this process is dominated by the $\mathcal O(m)$ applications of the valuated matroid intersection. Thus, the total computational time is $\mathcal O(m P T_\mathrm{VMI})$, where $T_{\mathrm{VMI}}$ denotes the time for solving the valuated matroid intersection problem, and $P$ is the number of Pareto optimal values.

Let $r$ denote the rank of the g-matroid $\M$, i.e., the largest cardinality among the feasible sets. Since $\sum_{i=1}^m \mu_i \leq r$ holds for each problem \eqref{prob:MLBu} that appears in the algorithm, the number of possible choices for $\mu$ is given by $\binom{r + m}{m} = \mathcal{O}(r^m)$, which is polynomial in $r$ when $m$ is a fixed constant. However, $P$ can be much smaller than this worst-case bound.  

\begin{algorithm}[H]
	\caption{MLB algorithm}
	\label{alg:MLB}
	\begin{algorithmic}[1] 
	\REQUIRE An M$^\natural$-convex function~$g$, and a lexicographic function~$\eta$
	\ENSURE Complete set $\mathcal X_{\mathrm{cP}}$ of Pareto optimal values 
    \STATE Find $X$ that lexicographically minimizes $\eta$ in
    $\Arg\min_Y\{g(Y)\mid Y\in\mathcal{P}\}$. 
    \STATE $\mathcal{X}_{\mathrm{cP}}\leftarrow \{(g(X),\eta(X)\}$ 
    \STATE $\mu\leftarrow \eta(X)$
    \WHILE{\eqref{prob:MLBuk} is feasible for some $k>0$}
        \FOR{$k=1,\ldots,m$}
            \STATE $\zeta_k \leftarrow \min_{Y} \{ g(Y) \mid Y\in \mathcal P\cap \mathcal{L}_k \}$
        \ENDFOR 
        \STATE Find the smallest index $\ell$ such that $\zeta_\ell=\min\{\zeta_k\mid k=1,\ldots,m\}$.
        \STATE Find $X$ that lexicographically minimizes $\eta$ in $\Arg\min_Y\{g(Y)\mid Y\in\mathcal{P}\cap\mathcal{L}_\ell\}$.
        \STATE $\mathcal X_{\mathrm{cP}} \leftarrow \mathcal X_{\mathrm{cP}}\cup \{(g(X),\eta(X))\}$
        \STATE $\mu\leftarrow\eta(X)$ 
	\ENDWHILE
    \RETURN $\mathcal X_{\mathrm{cP}}$
	\end{algorithmic}
\end{algorithm}


\section{Conclusions}
\label{sec:conclusion}

In this work, we considered three biobjective optimization problems over integer lattice points or g-matroids involving M$^\natural$-convex or M-convex functions, and showed that the entire Pareto optimal value set can be obtained in polynomial time for each of them. For the first two cases, which involve linear functions with binary coefficients, one may apply the restricted swap sequences introduced in~\cite{GT84} to derive more efficient algorithms. However, some of their results rely on the assumption that the cost of a swap is uniform across all bases. Extending these techniques beyond the linear setting, as considered in this paper, remains nontrivial. Another direction for future work is to increase the number of linear functions with binary coefficients, leading to problems with multiple objectives. However, the techniques developed in this study cannot be directly generalized to such cases, and further research is needed to handle such extensions. Finally, we expect that efficient algorithms can be developed for multiobjective optimization involving L-convex or L$^\natural$-convex functions. These functions form the other major pillar of discrete convex analysis, and it is therefore important to develop approaches that do not rely on the weighting method and are capable of recovering the set of Pareto optimal values.


\bibliographystyle{plain} 
\bibliography{journal-titles,references}

\end{document}